\documentclass{amsart}
\usepackage{amssymb,amsmath,amsrefs}
\usepackage[colorlinks=true]{hyperref}
\hypersetup{urlcolor=blue, citecolor=green}
\usepackage{graphicx} % Automatic driver for the marked PDF.
\usepackage{subcaption}
\usepackage{float}
\usepackage{tikz}
\usepackage[normalem]{ulem}
\definecolor{reviewteal}{RGB}{0,128,128}
\usetikzlibrary{arrows.meta,decorations.pathreplacing,calc}

\definecolor{stab}{RGB}{215,20,42}   % stable arcs (red)  -- A_1', A_2'
\definecolor{unst}{RGB}{28,20,230}   % unstable arcs (blue)-- U_1, U_2, C^u
\definecolor{bad}{RGB}{200,120,0}    % the problematic continuum / point

\graphicspath{{Imagens/}}
\theoremstyle{plain}

\newtheorem*{que}{Question}
\newtheorem{mainthm}{Theorem}

\newtheorem*{conj*}{Conjecture}
\newtheorem*{cor*}{Corollary}
\newtheorem{theorem}{Theorem}[section]
\newtheorem{thm}[theorem]{Theorem}

\newtheorem{corollary}[theorem]{Corollary}
\newtheorem{lemma}[theorem]{Lemma}

\newtheorem{claim}[theorem]{Claim}

\theoremstyle{definition}
\newtheorem*{def*}{Definition}
\newtheorem{remark}[theorem]{Remark}

\newtheorem{definition}[theorem]{Definition}

\newcommand{\im}{\operatorname{Im}}
\newcommand{\rea}{\operatorname{Re}}
\newcommand{\C}{\mathbb{C}}
\newcommand{\mc}{\mathcal}

\newcommand{\spine}{\operatorname{Spine}}

\newcommand{\T}{\mathbb{T}}

\renewcommand{\epsilon}{\varepsilon}

\newcommand{\Z}{\mathbb{Z}}
\newcommand{\N}{\mathbb{N}}
\newcommand{\R}{\mathbb{R}}
\newcommand{\eps}{\varepsilon}

\newcommand{\dist}{\operatorname{\textit{d}}}

\newcommand{\interior}{\operatorname{int}}

\newcommand{\sing}{\operatorname{Sing}}

\newcommand{\diam}{\operatorname{diam}}

\title[Continuum-wise hyperbolicity and pseudo-Anosov dynamics]{Continuum-wise hyperbolicity is exactly the pseudo-Anosov dynamics with spine singularities}

\author[R. Arruda, B. Carvalho, P. Oprocha,\\
and A. Sarmiento]{Rodrigo Arruda, Bernardo Carvalho, Piotr Oprocha,\\
and Alberto Sarmiento}

\thanks{2020 \emph{Mathematics Subject Classification}: Primary 37B45; Secondary 37D10.}
\keywords{cw-hyperbolicity, classification, pseudo-Anosov.}

\begin{document}

\begin{abstract}
We establish a complete structural classification for cw$_F$-hyperbolic surface homeomorphisms. Specifically, we prove that a surface homeomorphism is cw$_F$-hyperbolic if, and only if, it is a pseudo-Anosov homeomorphism  whose singularities consist exclusively of spines (1-prongs). Furthermore, we classify these systems up to topological conjugacy, showing that every such homeomorphism is conjugate to either an Anosov automorphism on the torus $\mathbb{T}^2$ or to a standard hyperelliptic quotient on $\mathbb{S}^2$ of a hyperbolic affine toral map. As a rigid consequence of this classification, we show that such dynamics are strictly obstructed on surfaces of genus greater than one, the Klein bottle, and the projective plane.
\end{abstract}

\maketitle

\section{Introduction}

The interplay between low-dimensional topology and dynamical systems has long been anchored by the Nielsen-Thurston classification of surface homeomorphisms \cite{BC, FLPKM, T}. At the heart of this theory lie the pseudo-Anosov maps, which serve as the standard geometric models for chaotic surface dynamics, constructed famously via Dehn twists \cite{P} or measured singular foliations \cite{FLP, HeHi}. Throughout, surfaces are compact, connected, and without boundary. Topologically, a pseudo-Anosov homeomorphism is defined by a pair of invariant, transverse, measured singular foliations that are uniformly expanded and contracted \cite{FLP}. The geometric rigidity of these maps is deeply tied to their singularity structure, which naturally corresponds to the zeros and poles of quadratic differentials in Teichmüller theory, the study of translation surfaces \cite{BL,La, MasurSmillie1993} and the geometry of hyperbolic 3-manifolds \cite{FLM}.

To maintain a consistent local product structure and ensure the uniqueness of the pseudo-Anosov representative within an isotopy class, the classical definition strictly mandates that all singularities must have three or more separatrices (prongs) \cite{FLM, FLPKM}. This topological restriction is beautifully mirrored in pure topological dynamics. In a foundational result, Hiraide \cite{Hi2} and Lewowicz \cite{L} independently proved that a surface homeomorphism is pseudo-Anosov if, and only if, it is expansive. This equivalence establishes a profound bridge between local geometry and global topology: expansiveness not only enforces the standard pseudo-Anosov foliation structure, but it also imposes strict topological obstructions. Specifically, it ensures that surfaces such as the two-dimensional sphere $\mathbb{S}^2$, the Klein bottle, and the projective plane cannot admit expansive homeomorphisms \cite{Hi2, L}.

While the dynamical systems community has actively explored generalized pseudo-Anosov maps -- 
such as through pruning theory and generalized pseudo-Anosov maps \cite{CH, CH2}, limits of pseudo-Anosov sequences \cite{BCGH}, and measurable pseudo-Anosov dynamics \cite{BCH} -- the strict topological obstructions imposed by expansiveness have remained a hard boundary.
Singularities with a single separatrix (spines, or 1-prongs), which correspond geometrically to simple poles of quadratic differentials \cite{MasurSmillie1993}, inherently break local expansiveness. Consequently, they have largely been treated as boundary anomalies rather than as objects with their own rigid dynamical classification.

%------------------------------

Despite being natural geometric objects, spines introduce severe dynamical pathologies through the lens of classical expansiveness. In an arbitrarily small neighborhood of a 1-prong, the local stable and unstable continua must physically ``turn around'' the singularity, forcing them to intersect at least twice. This local failure of unique transversality instantly destroys expansiveness, placing any pseudo-Anosov map with spines strictly outside the Hiraide-Lewowicz classification.

The geometric construction of such systems -- often via the hyperelliptic quotient of a torus -- is a classic tool in geometry, famously utilized by Thurston \cite{T} and rooted in the earlier work of Latt\`es (see the survey \cite{Milnor06} by Milnor). In the context of topological dynamics, P. Walters \cite{W} highlighted this  specific construction to provide an explicit counterexample showing that
a topological factor of an expansive homeomorphism need not inherit expansiveness. This phenomenon underscores the subtle interplay between global expansiveness and the local geometry of singular prongs, where the classical product structure collapses.

In subsequent years, systems exhibiting this behavior were explored from distinct viewpoints and  several generalized frameworks were deployed, including non-uniform hyperbolicity and entropy expansiveness \cite{PPV, PV}, as well as continuum-wise expansiveness and surface dendritations \cite{AAV, Artanom, ArD}. These investigations often treated maps with spines as boundary anomalies, 
reflecting the broader challenges in ergodic theory when one moves  beyond
the classical hyperbolic framework, a program famously initiated by Shub and Smale \cite{Mane1987, SS}. 

While these frameworks successfully illuminated various ergodic and local topological properties of maps with spines, a global topological characterization -- analogous to the Hiraide-Lewowicz theorem -- remained an open problem. This leads to a natural question: 

\begin{que}
If classical expansiveness is too restrictive to accommodate 1-prongs, what is the precise dynamical property that characterizes pseudo-Anosov homeomorphisms whose singularities are spines?
\end{que}

In this article, we establish that the precise dynamical framework governing these systems is continuum-wise ($cw_F$) hyperbolicity, a concept recently introduced to extend topological hyperbolicity beyond the expansive regime \cite{ACS, ACCV3}.
Building upon Kato's notion of continuum-wise expansiveness \cite{K1, K2} -- which permits non-trivial dynamical balls, thereby accommodating the ``turning'' of separatrices -- and incorporating a local product structure for these continua, $cw_F$-hyperbolicity captures the exact local geometry of 1-prongs without collapsing the global dynamics. Crucially, just as classical expansiveness acts as a rigid topological filter that mandates singularities of three or more prongs, we prove that $cw_F$-hyperbolicity imposes a different restriction on the singularities: it rigorously isolates the sub-class of pseudo-Anosov maps where higher-degree singularities are entirely forbidden. Our first main result establishes this equivalence.

\begin{mainthm}\label{Bcw}
A surface homeomorphism $f\colon S\to S$ is cw$_F$-hyperbolic if, and only if, $f$ is a pseudo-Anosov homeomorphism whose singularities  consist exclusively of spines. In particular, cw$_F$-hyperbolic surface homeomorphisms have a finite number of spines and are cw$_2$-hyperbolic.
\end{mainthm}

This strict exclusion of higher-degree singularities immediately unlocks a powerful topological rigidity. By the Poincaré-Hopf index theorem for measured foliations \cite{FLM, FLPKM}, the existence of pseudo-Anosov maps on surfaces of higher genus relies inherently on singularities with three or more prongs to balance the negative Euler characteristic. Because $cw_F$-hyperbolicity forces all singularities to be spines (and bounds local intersections to at most two), the allowable global topology collapses entirely to surfaces with non-negative Euler characteristics.

Our second main result provides the complete topological and geometric classification of these systems, proving that they are fundamentally tied to affine models on the sphere and to linear models on the torus:

\begin{mainthm}\label{A}
cw$_F$-hyperbolic surface homeomorphisms are topologically conjugate to either an Anosov automorphism of $\mathbb{T}^2$ or to the standard hyperelliptic quotient of an affine map of the form $x\to Ax+p$, where $A\in GL(2,\Z)$ is hyperbolic and $2p=0\hspace{-0.1cm}\mod\Z^2$, on the sphere $\mathbb{S}^2$. In particular, there are no cw$_F$-hyperbolic homeomorphisms on surfaces of genus greater than one, on the Klein bottle, or the projective plane.
\end{mainthm}

Together, these results bridge the gap between classical expansiveness and generalized pseudo-Anosov maps. They confirm that systems whose only singularities are spines are not mere boundary anomalies, but belong to a rigid, classifiable category of their own, governed by the exact same algebraic models that dominate the classical hyperbolic theory.

\vspace{0.3cm}
\noindent \textbf{Organization of the paper.}
In Section 2, we gather the necessary preliminary concepts, including the definitions of continuum-wise expansiveness and local product structures, and  perform a local dynamical analysis of spines, establishing the foundational geometric bounds required for our main arguments. Section 3 is devoted to the formal geometric properties of pseudo-Anosov homeomorphisms and singular foliations. In this section we prove Theorem \ref{Bcw}, demonstrating the exact equivalence between $cw_F$-hyperbolicity and pseudo-Anosov homeomorphisms with spines. Finally, in Section 4, we utilize this structural equivalence to deduce the global topological obstructions and prove the conjugacy classification stated in Theorem \ref{A}.  The paper concludes with an appendix, where we provide a proof that the lifts of a cw$_F$-hyperbolic homeomorphism by a finite covering map are cw$_F$-hyperbolic.

\section{Finiteness of Spines}

In this section, we prove that cw$_F$-hyperbolic homeomorphisms have at most a finite number of spines. First, we state necessary definitions to describe precisely cw$_F$-hyperbolicity. Its definition is the combination of several definitions presented in distinct previous articles, which are stated and discussed below.

\begin{definition}[Local stable/unstable sets]
Let $(X,d)$ be a compact metric space and $f\colon X\to X$ be a homeomorphism.
We consider the \emph{c-stable set} of $x\in X$ as the set 
$$W^s_{c}(x):=\{y\in X; \,\, d(f^k(y),f^k(x))\leq c \,\,\,\, \textrm{for every} \,\,\,\, k\geq 0\}$$
and the \emph{c-unstable set} of $x$ as the set 
$$W^u_{c}(x):=\{y\in X; \,\, d(f^k(y),f^k(x))\leq c \,\,\,\, \textrm{for every} \,\,\,\, k\leq 0\}.$$
\end{definition}

Using local stable/unstable sets we can define the dynamical ball of points in the space and expansiveness, as defined in \cite{Utz}.

\begin{definition}[Expansiveness]
For each $x\in X$ and $c>0$, let $$\Gamma_{c}(x)=W^u_{c}(x)\cap W^s_{c}(x)$$ be the \emph{dynamical ball} of $x$ with radius $c$. We say that $f$ is \emph{expansive} if there exists $c>0$ such that $$\Gamma_c(x)=\{x\} \,\,\,\,\,\, \text{for every} \,\,\,\,\,\, x\in X.$$ A constant $c$ satisfying the above definition is called an expansivity constant of $f$.
\end{definition}

The cw-expansive homeomorphisms can have non-trivial dynamical balls; they can contain even Cantor sets inside arbitrarily small dynamical balls, but still increase non-trivial continua of the space as in the following definition.

\begin{definition}[cw-expansiveness]\label{cwexp}
We say that $f$ is continuum-wise expansive (or simply $cw$-expansive) if there exists $c>0$ such that $$\Gamma_c(x) \quad \text{ is totally disconnected }$$ for every $x\in X$.
\end{definition}

This was introduced by H. Kato in \cites{K1,K2} and, informally speaking, means that the diameter of iterates of non-trivial continua must increase by a uniform amount over time. The class of cw-expansive homeomorphisms is much wider than the class of expansive homeomorphisms. In fact, many systems which appear in chaotic topological dynamics and continuum theory are cw-expansive but are not expansive. Kato discussed some of these examples in \cites{K1,K2} and also explored similarities and differences between expansive and cw-expansive systems. A classification of cw-expansive homeomorphisms is far from being obtained and new examples of such systems have appeared in the last few years \cite{AAV,Artanom,ArD,Artrobn,APV,ACC,ACCV2,CC}. To define cw$_F$-expansiveness we need to consider the local stable/unstable continua as follows.

\begin{definition}[Local stable/unstable continua]
We denote by $C^s_c(x)$ the $c$-stable continuum of $x$, that is, the connected component of $x$ in $W^s_{c}(x)$, and denote by $C^u_c(x)$ the $c$-unstable continuum of $x$, that is, the connected component of $x$ in $W^u_{c}(x)$.
\end{definition}

There is an important difference between the local stable/unstable continua $C^s_{\eps}(x)/ C^u_{\eps}(x)$, which in many examples of surface homeomorphisms are arcs, and the local stable/unstable sets $W^s_{\eps}(x)/ W^u_{\eps}(x)$, which can be even the union of a Cantor set of disjoint arcs \cite{ArD}. This difference translates to differences between cw-expansiveness and cw$_F$-expansiveness, defined as follows.

\begin{definition}[cw$_F$-expansiveness]\label{cwf}
A $cw$-expansive homeomorphism is said to be cw$_F$-expansive if there exists $c>0$ such that $$\#(C^s_c(x)\cap C^u_c(x))<\infty \quad \text{ for every } \quad x\in X.$$ Analogously, $f$ is said to be $cw_N$-expansive if there is $c>0$ such that $$\#(C^s_c(x)\cap C^u_c(x))\leq N \quad \text{ for every } \quad x\in X.$$    
\end{definition}

Cw$_F$-expansiveness was defined in \cite{ArD} to describe dynamics around a spine singularity in terms of expansiveness properties. At first glance, the structure of two arbitrarily close intersections between stable/unstable leaves close to a spine, and single intersections around regular points, suggests that the dynamics is 2-expansive; recall that $N$-expansiveness was defined in \cite{Mor} (and further explored in \cites{Artrobn,APV,CC,CC2,LZ}) as the existence of a bound N for the number of points allowed in all dynamical balls of the system. But the difference between the local stable/unstable sets and local stable/unstable continua plays an important role, while the existence of the Cantor sets explained above contradicts N-expansiveness (see Section 2.2.1 in \cite{ArD}). Now we are able to define cw$_F$-hyperbolicity.

\begin{definition}[cw-hyperbolicity]\label{cwfh}
We say that $f$ satisfies the $cw$-local-product-structure if for each $\eps>0$ there exists $\delta>0$ such that $$C^s_\eps(x)\cap C^u_\eps(y)\neq \emptyset \quad \text{ whenever } \quad \dist(x,y)< \delta.$$ The cw-expansive homeomorphisms (resp.\ cw$_F$, cw$_N$) satisfying the cw-local-product-structure are called cw-hyperbolic (resp.\ cw$_F$, cw$_N$).
\end{definition}

Cw-hyperbolic systems share several important properties with the topologically hyperbolic ones, such as the L-shadowing property \cite{ACCV} and a spectral decomposition theorem \cite{A,ACCV3,S}, but a few important differences exist, %are noted on pseudo-Anosov diffeomorphisms such as on 
as revealed by the previously mentioned
Walters's example on $\mathbb{S}^2$ (see further results in \cite{CR,COR}). In fact, the examples induced on the sphere by a hyperbolic matrix in $SL(2,\Z)$ through the quotient of identifying antipodal points (which is a hyperelliptic quotient) are examples of cw$_2$-hyperbolic homeomorphisms.

Some results in the direction of a classification of cw$_F$-hyperbolic surface homeomorphisms were obtained in \cite{ACS}. In what follows, we describe these results and use them to prove that cw$_F$-hyperbolic surface homeomorphisms $f\colon S\to S$ have a finite number of spines. The first fact is that local stable/unstable continua are locally connected \cite{ACS}*{Lemma 2.6}, that is, there exists $\eps>0$ such that $C^{\sigma}_{\eps}(x)$ is locally connected for every $x\in X$ and $\sigma\in\{s,u\}$. In fact, these continua are arcs by \cite{ACS}*{Proposition 2.13}. Thus, for each $x\in X$ and $y\in C^{\sigma}_{\eps}(x)$ we can consider the unique arc $\sigma(x,y;x)$ in $C^\sigma_\varepsilon(x)$ connecting $x$ to $y$.

\begin{definition}[Number of separatrices]
For $y,z\in C^\sigma_\varepsilon(x)\setminus\{x\}$ we define
the relation
$
y\sim z$
 if $\sigma(x,y;x)\cap\sigma(x,z;x)\supsetneq\{x\}.$
The number of $\sigma$-separatrices at $x$ is defined as
$$P^\sigma(x)=\#\big((C^\sigma_\varepsilon(x)\setminus\{x\})/\sim\big).$$
\end{definition}
For cw$_F$-hyperbolic surface homeomorphisms the numbers of stable and unstable separatrices coincide, that is,
$$P^s(x)=P^u(x) \,\,\,\,\,\, \text{for every} \,\,\,\,\,\, x\in S$$
and we denote their common value by $p(x)$ (see \cite{ACS}*{Corollary 2.11}). The number of separatrices $p(x)$ is at most $2$ (see \cite{ACS}*{Lemma 2.12}) which is a consequence of the fact that local stable/unstable continua are necessarily arcs. A point $x\in S$ such that $p(x)=1$ is called a \emph{spine}, and the set $\spine(f)$ is defined as the set of spines of $f$. Points $x\in S$ that are not spines are called \emph{regular} and satisfy $p(x)=2$. 

The cw-local-product-structure ensures that intersections between stable and unstable continua are topologically transversal when the point of intersection is regular (see Lemma 2.16 of \cite{ACS}). 
%\textcolor{red}{Rodrigo, incluir a figura da transversalidade aqui.} 
The following is the main result of this section.

\begin{thm}\label{finitespines}
    If $f$ is a cw$_F$-hyperbolic surface homeomorphism, then $\spine(f)$ is a finite set.
\end{thm}

\begin{proof}
Suppose for a contradiction that the set of spines is infinite. Then we have a sequence of distinct spines $(x_i)_{i\in \N}$, and this sequence accumulates at some $x\in S$. Up to passing to a subsequence, we may assume that $(x_i)_{i\in \N}$ converges to $x$. It is a consequence of \cite{ACS}*{Lemma 3.7} that $x$ is not a spine because around each spine there is a bi-asymptotic sector ensuring that spines are isolated from other spines. Cw$_F$-expansiveness ensures the existence of $\eps' \in (0,\eps)$ such that $C^s_{\eps'}(x)\cap C^u_{\eps'}(x)=\{x\}$ (see Lemma 2.10 of \cite{ACS}). Let $\delta'\in(0,\eps')$ and $\gamma\in(0,\eps')$ be the constants of $cw$-local-product-structure and uniform diameter of local stable/unstable continua for $\eps'>0$, respectively.
Choose a closed disk $D_0$ around $x$, with $\diam D_0<\gamma$, such that $D_0\cap C_{\eps'}^\sigma(x)$, $\sigma=s,u$, are arcs with endpoints in $\partial D_0$. By \cite{ACS}*{Lemma 2.4}, choose $D_0$ small enough that every stable or unstable continuum $K\subset D_0$ satisfies
$$K\subset C_{\eps'}^\sigma(y)\qquad\text{for every } y\in K,$$
with the corresponding choice of $\sigma$. Pass to a subsequence in one component $Q$ of $D_0\setminus(C_{\eps'}^s(x)\cup C_{\eps'}^u(x))$. 
%Apply local product structure at a smaller scale so that the arcs below lie in $D_0$.
 The cw-local-product-structure ensures that, for $i$ sufficiently large, we have
$$C^s_{\eps'}(x_i)\cap C^u_{\eps'}(x) \neq \emptyset \,\,\,\,\,\, \text{and} \,\,\,\,\,\, C^u_{\eps'}(x_i)\cap C^s_{\eps'}(x)\neq \emptyset.$$
Let $l^u_{x_i}$ be an arc in $C^u_{\eps'}(x_i)$ connecting $x_i$ to a point $r_1\in C^s_{\eps'}(x)$ intersecting $C^s_{\eps'}(x)$ only at $r_1$. Let $l^s_{x_i}$ be an arc in $C^s_{\eps'}(x_i)$ connecting $x_i$ to a point $r_2\in C^u_{\eps'}(x)$ intersecting $C^u_{\eps'}(x)$ only at $r_2$. 
Let $A_1\subset C^s_{\eps'}(x)$ be the arc connecting $x$ to $r_1$ and $l_x^u\subset C^u_{\eps'}(x)$ be the arc connecting $x$ to $r_2$. If $l^\sigma_{x_i}$ met $C_{\eps'}^\sigma(x)$, applying the preceding inclusion to $l^\sigma_{x_i}\cup(D_0\cap C_{\eps'}^\sigma(x))$ would give $x_i\in C_{\eps'}^\sigma(x)$, contrary to $x_i\in Q$. Thus both arcs lie in $\overline Q$. Let $q$ be the first point of $l^s_{x_i}\cap l^u_{x_i}$ along $l^s_{x_i}$ from $r_2$ towards $x_i$. Let $A_2$ and $l_q^u$ be the subarcs joining $r_2$ to $q$ in $l^s_{x_i}$ and $q$ to $r_1$ in $l^u_{x_i}$, respectively. Then $A_2\cap l^u_{x_i}=\{q\}$.

Thus, this configuration forms a topological rectangle $R$ with $\diam R<\gamma$ delimited by the subarcs 
$$A_1\subset C^s_{\eps'}(x), \,\,l_x^u\subset C^u_{\eps'}(x), \,\, A_2\subset C^s_{\eps'}(x_i), \,\, \text{and} \,\, l_{q}^u\subset C^u_{\eps'}(x_i).$$
Up to choosing a subsequence, we may assume that $i=1$ (see Figure \ref{fig:toprectangle}).

    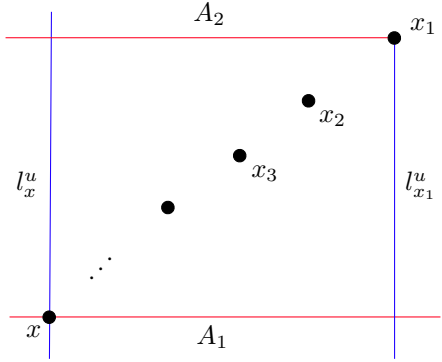
\begin{figure}[H]
        \centering
        \begin{tikzpicture}[x=0.75pt,y=0.75pt,yscale=-1,xscale=1]
        
        %Straight Lines [id:da030976747887525025] 
        \draw [color={rgb, 255:red, 30; green, 0; blue, 255 }  ,draw opacity=1 ]   (239,58) -- (238,232) ;
        %Straight Lines [id:da3239738110209489] 
        \draw [color={rgb, 255:red, 255; green, 5; blue, 36 }  ,draw opacity=1 ]   (434,211) -- (218,211) ;
        %Straight Lines [id:da6947891632530371] 
        \draw [color={rgb, 255:red, 255; green, 5; blue, 36 }  ,draw opacity=1 ]   (411,71) -- (216,71) ;
        %Straight Lines [id:da8332788695771971] 
        \draw [color={rgb, 255:red, 30; green, 0; blue, 255 }  ,draw opacity=1 ]   (411,71) -- (411,232) ;
        %Shape: Circle [id:dp8061311924146355] 
        \draw  [fill={rgb, 255:red, 0; green, 0; blue, 0 }  ,fill opacity=1 ] (234.75,211.13) .. controls (234.75,209.4) and (236.15,208) .. (237.88,208) .. controls (239.6,208) and (241,209.4) .. (241,211.13) .. controls (241,212.85) and (239.6,214.25) .. (237.88,214.25) .. controls (236.15,214.25) and (234.75,212.85) .. (234.75,211.13) -- cycle ;
        %Shape: Circle [id:dp8194059042877077] 
        \draw  [fill={rgb, 255:red, 0; green, 0; blue, 0 }  ,fill opacity=1 ] (413.78,72.28) .. controls (413.14,73.89) and (411.32,74.67) .. (409.72,74.03) .. controls (408.11,73.39) and (407.33,71.57) .. (407.97,69.97) .. controls (408.61,68.36) and (410.43,67.58) .. (412.03,68.22) .. controls (413.64,68.86) and (414.42,70.68) .. (413.78,72.28) -- cycle ;
        %Shape: Circle [id:dp6867580057606527] 
        \draw  [fill={rgb, 255:red, 0; green, 0; blue, 0 }  ,fill opacity=1 ] (370.78,103.78) .. controls (370.14,105.39) and (368.32,106.17) .. (366.72,105.53) .. controls (365.11,104.89) and (364.33,103.07) .. (364.97,101.47) .. controls (365.61,99.86) and (367.43,99.08) .. (369.03,99.72) .. controls (370.64,100.36) and (371.42,102.18) .. (370.78,103.78) -- cycle ;
        %Shape: Circle [id:dp953439951210983] 
        \draw  [fill={rgb, 255:red, 0; green, 0; blue, 0 }  ,fill opacity=1 ] (336.28,131.28) .. controls (335.64,132.89) and (333.82,133.67) .. (332.22,133.03) .. controls (330.61,132.39) and (329.83,130.57) .. (330.47,128.97) .. controls (331.11,127.36) and (332.93,126.58) .. (334.53,127.22) .. controls (336.14,127.86) and (336.92,129.68) .. (336.28,131.28) -- cycle ;
        %Shape: Circle [id:dp8435302503472618] 
        \draw  [fill={rgb, 255:red, 0; green, 0; blue, 0 }  ,fill opacity=1 ] (300.28,157.28) .. controls (299.64,158.89) and (297.82,159.67) .. (296.22,159.03) .. controls (294.61,158.39) and (293.83,156.57) .. (294.47,154.97) .. controls (295.11,153.36) and (296.93,152.58) .. (298.53,153.22) .. controls (300.14,153.86) and (300.92,155.68) .. (300.28,157.28) -- cycle ;
        
        % Text Node
        \draw (280,179.41) node [anchor=north west][inner sep=0.75pt]  [rotate=-96.29]  {$\ddots $};
        % Text Node
        \draw (417.5,60) node [anchor=north west][inner sep=0.75pt]    {$q$};
        % Text Node
        \draw (371.88,106.03) node [anchor=north west][inner sep=0.75pt]    {$x_{2}$};
        % Text Node
        \draw (338,133.9) node [anchor=north west][inner sep=0.75pt]    {$x_{3}$};
        % Text Node
        \draw (220,136.9) node [anchor=north west][inner sep=0.75pt]    {$l_{x}^{u}$};
        % Text Node
        \draw (415,136.9) node [anchor=north west][inner sep=0.75pt]    {$l_{q}^{u}$};
        % Text Node
        \draw (310.5,214.4) node [anchor=north west][inner sep=0.75pt]    {$A_{1}$};
        % Text Node
        \draw (309,52) node [anchor=north west][inner sep=0.75pt]    {$A_{2}$};

        \draw (225,214) node [anchor=north west][inner sep=0.75pt]  {$x$};
        \end{tikzpicture}
        \caption{Topological rectangle $R$.}
        \label{fig:toprectangle}
    \end{figure}
We have $R\subset\overline Q$ and $V\cap Q\subset\interior(R)$ for some neighborhood $V$ of $x$. Thus, after discarding finitely many terms, $x_i\in\interior(R)$ for every $i\geq2$.

For each $y\in R$, let $C^u_R(y)$ denote the connected component of $C^u_{\eps'}(y)\cap R$ that contains $y$ and $C^s_R(y)$ denote the connected component of $C^s_{\eps'}(y)\cap R$ that contains $y$. Applying the preceding inclusion to the union of two intersecting stable or unstable continua gives, with analogous notation for every closed subdisk $D\subset R$,
$$z\in C_D^\sigma(y)\quad\Longrightarrow\quad C_D^\sigma(z)=C_D^\sigma(y),\qquad \sigma=s,u.$$

%Fix $n\in \N$ and the spine $x_n$. Then $C^u_R(x_n)$ intersects either $A_1$ or $A_2$, but not both. Assume that it intersects $A_1$ at the point $y_n$, the case which it intersects $A_2$ is analogous.

\begin{claim}\label{claim}
        The set of points $y\in A_1$ such that $C^u_R(y)$ intersects $A_2$ is closed in $A_1$.
\end{claim}    

\begin{proof}[Proof of the Claim]
Suppose we have a sequence $(y_n)_{n\in\N}\subset A_1$ converging to $y\in A_1$ such that 
$$C^u_R(y_n)\cap A_2\neq \emptyset \,\,\,\,\,\, \text{for every} \,\,\,\,\,\, n\in\N.$$
For each $n\in\N$, choose $z_n\in C^u_R(y_n)\cap A_2$.
%Since $y\in A_1$, it is clear that $C^u_R(y)\cap A_1\neq \emptyset$.
The semi-continuity of the map $y\mapsto C^u_{R}(y)$ (see page 15 and Theorem 6.7.1 of \cite{ArD}) ensures that any accumulation point $z\in A_2$ of $(z_n)_{n\in\N}$ satisfies $z\in C^u_R(y)$. This proves the claim.
%every neighborhood of $C^u_R(y)$ must intersect $A_2$. Since $C^u_R(y)$ is closed, we have that $C^u_R(y)\cap A_2\neq \emptyset$. 
\end{proof}

The cw-local-product-structure ensures that for sufficiently large $i\in\N$ we have that $C^u_{\eps'}(x_i)$ intersects $A_1$ before it intersects $A_2$ (in the order given by the parametrization of the arc $C^u_{\eps'}(x_i)$ starting at $x_i$), which in turn ensures that $C^u_{R}(x_i)\cap A_2=\emptyset$, given the transversality between $C^u_{\eps'}(x_i)$ and $A_1$. This proves the existence of points $z\in A_1$ such that $C^u_{R}(z)$ connects $z$ and a spine in the interior of $R$, which ensures in particular that $C^u_{R}(z)\cap A_2=\emptyset$. By considering the closure of the connected component of $z$ in the complement of the set of Claim \ref{claim} we obtain a non-trivial sub-arc $A_1'\subset A_1$ such that $C^u_R(y)\cap A_2\neq\emptyset$ if, and only if, $y$ belongs to the boundary of $A_1'$. In particular, $C^u_R(y)\cap A_2=\emptyset$ for every $y$ in the interior of $A_1'$.

Let $y_1$ and $y_2$ be the points of the boundary of $A_1'$. For $j=1,2$, let $U_j\subset C_R^u(y_j)$ be the subarc from $y_j$ to its first intersection $w_j$ with $A_2$. Except for the components containing $l_x^u$ and $l_q^u$, each contact of an unstable component with $A_1\cup A_2$ is an endpoint, by transversality. These two exceptional unstable components meet $A_1$ only at $x$ and $r_1$, respectively. Thus $C_R^u(y_1)\neq C_R^u(y_2)$, so $U_1\cap U_2=\emptyset$ and $w_1\neq w_2$. We obtain a rectangle $R'\subset R$ with boundary $A_1'\cup U_1\cup A_2'\cup U_2$, where $A_2'$ is the nontrivial subarc joining $w_1$ to $w_2$ in $A_2$.
This rectangle satisfies $C^u_{R'}(y)\cap A_2'=\emptyset$ for every $y$ in the interior of $A_1'$.

%if, and only if, $y$ belongs to the arcs $C^u_{R}(y_1)$, $C^u_{R}(y_2)$, and $A_2'$.
%for every point $y\in A_1'$, we have that $C^u_{R'}(y)\cap A_2=\emptyset$, (see Fig...).
For each $i\in\{1,2\}$, let $B_i\subset\interior(R')$ be the set of points $y\in\interior(R')$ such that $C^u_{R'}(y)\cap A_i'\neq\emptyset$. If $x_1\in R'$, denote by $J\subset R'$ the possibly degenerate unstable subarc
joining $x_1$ to the corner $q$. Otherwise, put $J=\emptyset$.
Set
$B=\interior(R')\setminus J$ and
$\widehat B_i=B_i\setminus J$ for $i\in\{1,2\}.$
The set $B$ remains connected, since the arc $J$, whenever
it enters the interior of $R'$, is attached to its boundary
only at $q$ and ends at the interior point $x_1$.

Any unstable component meeting both $A_1'$ and $A_2'$ must
pass through an endpoint of $A_1'$. Such a boundary component
can enter the interior of $R'$ only along $J$. Hence
$B_1\cap B_2\subset J$.
Since $\diam R'<\gamma$, every $C_{R'}^u(y)$ with $y\in B$ meets $\partial R'$. If it met $U_j$, then $C_{R'}^u(y)=C_{R'}^u(y_j)$, whose intersection with $\interior(R')$ is contained in $J$, a contradiction. Hence it meets $A_1'\cup A_2'$.
Therefore
$$B=\widehat B_1\cup\widehat B_2 \qquad\text{and}\qquad\widehat B_1\cap\widehat B_2=\emptyset.$$
The sets $\widehat B_1$ and $\widehat B_2$ are relatively
closed in $B$ by upper semicontinuity and both are nonempty
by transversality at interior points of $A_1'$ and $A_2'$,
away from $J$. This contradicts the connectedness of $B$.
% It follows that $B_1\cap B_2=\emptyset$, but this contradicts the connectedness of $\interior(R')$ since $\interior(R')=B_1\cup B_2$, with $B_1$ and $B_2$ being non-empty closed subsets in $\interior(R')$. 
This contradiction proves that the set of spines is finite. 
\end{proof}

The following result generalizes \cite{ACS}*{Theorem 1.1} by removing the hypothesis on the number of spines being finite.

\begin{corollary}\label{cor:cw2}
Any cw$_F$-hyperbolic surface homeomorphism is cw$_2$-hyperbolic.
\end{corollary}

\begin{proof}
If $f\colon S\to S$ is a cw$_F$-hyperbolic surface homeomorphism, then Theorem \ref{finitespines} ensures it has at most a finite number of spines and \cite{ACS}*{Theorem 1.1} ensures it is cw$_2$-hyperbolic.
\end{proof}

\section{Singular stable/unstable foliations}

In this section, we prove  Theorem \ref{Bcw}, the equivalence between cw$_F$-hyperbolicity and the pseudo-Anosov property with singularities consisting exclusively of spines.
%connected stable/unstable sets form a pair of transversal $C^0$ singular foliations with a finite number of singularities and conclude that $f$ is a pseudo-Anosov homeomorphism with spines. 
We begin with the definition of local singular charts and a $C^0$ singular foliation. 

\begin{definition}[Local singular charts] For each $p\in \N$, let $\pi_p:\C \to \C$ be the map which sends $z$ to $z^p$ and define connected open subsets $\mc{D}_p$ $(p=1,2,\ldots)$ of $\C$ by $$\mc{D}_2=\{z\in \C : |\rea(z)|<1 ,|\im(z)|<1\},$$ $$\mc{D}_1=\pi_2(\mc{D}_2), \quad \text{ and } \quad \mc{D}_p = \pi_p^{-1}(\mc{D}_1).$$ 
By the definition of $\pi_p$ and the construction of $\mc{D}_p$, we have that $\pi_p:\mc{D}_p\to \mc{D}_1$ is a $p$-fold branched covering map for every $p\in \N$. 
Denote by $\mc{H}_2$ and $\mc{V}_2$ the horizontal and vertical foliations on $\mc{D}_2$ respectively. We define decompositions $\mc{H}_1$ and $\mc{V}_1$ of $\mc{D}_1$ as the projections of $\mc{H}_2$ and $\mc{V}_2$ by $\pi_2:\mc{D}_2\to \mc{D}_1$, respectively. For $p\geq 3$, we define decompositions $\mc{H}_p$ and $\mc{V}_p$ of $\mc{D}_p$ as the lifting of $\mc{H}_1$ and $\mc{V}_1$ by $\pi_p:\mc{D}_p\to \mc{D}_1$, respectively.
\end{definition}

\begin{definition}[$C^0$ singular foliation]\label{singfol}
A decomposition $\mc{F}$ of $M$ is called a $C^0$ singular foliation if each of its leaves is path-connected and for each $x\in M$ there exist $p(x)\in\N$ and a $C^0$ chart $\phi_x\colon U_x\to\C$ around $x$ satisfying:
\begin{enumerate}
    \item $\phi_x(x)=0$
    \item $\phi_x(U_x)=\mc{D}_{p(x)}$,
    \item if $U_x\cap L\neq \emptyset$ for some $L\in\mc{F}$, then $\phi_x$ sends each connected component of $U_x\cap L$ onto some element of $\mc{H}_{p(x)}$.
\end{enumerate}
\end{definition}

The number $p(x)$ is called the number of separatrices at $x$. A regular point is a point $x\in S$ with $p(x)=2$, and $x$ is a singular point if $p(x)\neq 2$. We denote the set of singular points of $\mc{F}$ by $\sing(\mc{F})$ and denote by $\mc{RF}$ the $C^0$ foliation on $S\setminus \sing(\mc{F})$ obtained from $\mc{F}$ by removing its singular points. A $C^0$ singular foliation $\mc{F}$ on $S$ is called minimal if every leaf of $\mc{RF}$ is dense in $S$.

In what follows, $f\colon S\to S$ is a cw$_F$-hyperbolic surface homeomorphism, $c$ is a cw-expansive constant of $f$, and $\eps\in(0,\frac{c}{2})$.  
%Let $\spine (f)$ denote the set of spines of $f$ and assume this is a finite set. 
We define $\mc{F}^\sigma_f$ as $$\mc{F}^\sigma_f=\{C^\sigma(x); \,\,x\in S\} \,\,\,\,\,\,\,\, (\sigma=s,u)$$
where
$$C^s(x):=\bigcup_{n\in\N}f^{-n}(C^s_{\eps}(f^n(x))) \,\,\,\,\,\, \text{and} \,\,\,\,\,\, C^u(x):=\bigcup_{n\in\N}f^{n}(C^u_{\eps}(f^{-n}(x))).$$
We say that $C^s(x)$ is the connected stable set of $x$ and that $C^u(x)$ is the connected unstable set of $x$. This defines decompositions of $S$ into connected stable/unstable sets of $f$. The following is the main result of this section:

\begin{theorem}\label{foliation}
If $f\colon S\to S$ is a cw$_F$-hyperbolic surface homeomorphism, then the decomposition $\mc{F}^\sigma_f$ $(\sigma=s,u)$ satisfies:
\begin{enumerate}
    \item\label{fol:1} $\mc{F}^\sigma_f$ is a $C^0$ singular foliation where singularities are the spines of $f$,
    \item\label{fol:2} each leaf of $\mc{F}^\sigma_f$ is homeomorphic, in its leaf topology, to either $\mathbb{R}$ or $[0,\infty)$,
    \item\label{fol:3} $\mc{F}^s_f$ is transverse to $\mc{F}^u_f$,
    \item\label{fol:4} $\mc{F}^\sigma_f$ is minimal.
\end{enumerate}
% Furthermore, 
% %if $\spine(f)\neq \emptyset$, then 
% there exists an Anosov automorphism $\tilde f$ of $\mathbb{T}^2$ and a factor map $\pi\colon (\mathbb{T}^2,\tilde f) \to (S,f)$ such that $\pi$ is at most 2-to-1 and $\# \pi^{-1}(p)=1$ if, and only if, $p\in \spine(f)$.
Furthermore, if $\spine(f)\neq\emptyset$, then there exists an Anosov automorphism
$\widetilde f$ of $\T^2$ and a branched double covering factor map $\pi\colon(\T^2,\widetilde f)\to(S,f)$ such that
$\pi$ is at most $2$-to-$1$ and $\#\pi^{-1}(p)=1$ if, and only if, $p\in\spine(f)$. If instead
$\spine(f)=\emptyset$, then $\mc{F}^s_f,\mc{F}^u_f$ are nonsingular and $f$ is topologically
hyperbolic,  hence conjugate to an Anosov automorphism
of $\T^2$.
\end{theorem}

% {\blue
% Before we prove Theorem~\ref{foliation}, we state the following immediate corollary.
% \begin{corollary}
%     Let $f$ be a pseudo-Anosov surface homeomorphism and let $A\in SL(2,\Z)$ define the Anosov automorphism
% $\widetilde f_A$ provided jointly by Theorem~\ref{foliation} and Theorem~\ref{Bcw}. Then the dilatation $\lambda$ of $f$ equals the
% spectral radius $\Lambda$ of $A$; equivalently, the topological entropy of $f$ is
% $h(f)=\log\lambda=\log\Lambda$.
% \end{corollary}
% \begin{proof}
% By Theorem~\ref{A}, $f$ is pseudo-Anosov with spines of dilatation
% $\lambda$
% By Theorem~\ref{foliation}, $f$ is a topological factor of $\widetilde f_A$
% through a map $\pi$ with at most two points in each fibre (or is a conjugacy). It is well known that finite-to-one factor maps preserve topological entropy (e.g. see
% \cite{W}), hence $h(f)=h(\widetilde f_A)=\log\Lambda$, the last equality being the entropy of a
% linear Anosov automorphism. 

% {\red [Here we need argument about proper lifting of measures]}

% Therefore $\lambda=\Lambda$ completing the proof.
% \end{proof}

% }

To construct the local charts around points of $S$, we need to recall results in \cite{ACS} relating spines and bi-asymptotic sectors. These sectors were introduced in \cite{APV} for $N$-expansive homeomorphisms on surfaces and were defined as disks bounded by the union of a local stable arc and a local unstable arc intersecting only twice. It is proved that every spine is contained in the interior of a (regular) bi-asymptotic sector \cite{ACS}*{Lemma 3.7}, whose stable/unstable arcs forming it point outward from its interior at their intersections (see \cite{ACS}*{Definition 3.1}). It is also known that bi-asymptotic sectors always contain spines in their interior and that this spine is unique when the sector is regular \cite{ACS}*{Proposition 3.6}. Now we proceed to the proof of the theorem.

\begin{proof}[Proof of Theorem \ref{foliation}]
First, note that each leaf of $\mc{F}^\sigma_f$ is path-connected. This follows by noting that each local stable/unstable continuum is path-connected (see \cite[Corollary 2.7]{ACS}) and using the definition of $C^{\sigma}(x)$ as the union of iterates of local stable/unstable continua. To prove that $\mc{F}^\sigma_f$ is a $C^0$ singular foliation we must exhibit local charts $\phi_x\colon U_x\to\C$ around each point $x\in S$ as in Definition \ref{singfol}. If $x$ is a regular point, then there is a neighborhood $V_x$ of $x$ containing only regular points (since $\spine(f)$ is finite) and such that the local stable/unstable continua of every point in $V_x$ are arcs that separate $V_x$ (see \cite{ACS}*{Propositions 2.13 and Lemma 2.14}). For each $w\in V_x$ let $L^{\sigma}_w$ be an arc in $C^{\sigma}_{\eps}(w)$ containing $w$ and connecting two points on the boundary of $V_x$ (which are the only intersections between $L^{\sigma}_w$ and the boundary of $V_x$). Using the cw-local-product-structure with a sufficiently small constant, we obtain
%Let 
$U_x\subset V_x$, a neighborhood of $x$ such that 
$$L^s_y\cap L^u_z\neq\emptyset \,\,\,\,\,\, \text{whenever} \,\,\,\,\,\, y,z\in U_x.$$ Let $L^{\sigma}_{U_x}$ be the connected component of $L^{\sigma}_x\cap U_x$ that contains $x$. Thus, if $(y,z)\in L^s_{U_x}\times L^u_{U_x}$, then $L^u_y\cap L^s_z\neq\emptyset$. Note that $L^u_y\cap L^s_z$ is a singleton, since otherwise they would form a bi-asymptotic sector and $V_x$ would contain a spine (see \cite{ACS}*{Proposition 3.6}), contradicting the choice of $V_x$. Thus, there is a well-defined map $$h\colon L^s_{U_x}\times L^u_{U_x}\to V_x$$ defined by $h(y,z)=L^u_y\cap L^s_z$. The lack of spines in $V_x$ also ensures the injectivity of $h$. Continuity of $h$ follows from the same argument as in \cite{APV}*{Lemma 3.11}, %as well as the fact that $h$ is an open map. 
because a limit of $\eps$-stable/unstable continua (in Hausdorff distance) is itself an $\eps$-stable/unstable continuum. We just proved that $h$ is a homeomorphism onto its image.
Thus, we can use the inverse of $h$ to define the desired $C^0$-chart from a neighborhood of $x$ to $\mc{D}_2$.

Now assume that $x\in\spine(f)$ and recall that it is proved in \cite{ACS}*{Lemma 3.7} that $x$ belongs to the interior of a regular bi-asymptotic sector. The structure of local stable/unstable continua inside any regular bi-asymptotic sector is exactly the same as around the spines of the pseudo-Anosov diffeomorphism of $\mathbb{S}^2$. Indeed, there is a total order among the local stable/unstable continua inside these sectors (see \cite{ACS}*{Lemma 3.4}). From this and the continuity in \cite{ACS}*{Lemma 3.5}, it is easy to obtain a local chart from any regular bi-asymptotic sector to $\mc{D}_1$. This proves that $\mc{F}^\sigma_f$ is a $C^0$ singular foliation where singularities are the spines of $f$ and finishes the proof of \eqref{fol:1}.

To prove \eqref{fol:2} and \eqref{fol:3}, we note that similar statements were already proved for local stable/unstable continua. If $x$ is a regular point, then there exists a homeomorphism $h\colon[-1,1]\to C^{\sigma}_{\eps}(x)$ such that $h^{\sigma}(0)=x$, and if $x$ is a spine, then there exists a homeomorphism $h^{\sigma}\colon[0,1]\to C^{\sigma}_{\eps}(x)$ such that $h(0)=x$. Also, for $\eps>0$, there exists $\delta\in(0,\eps)$ such that stable/unstable continua with diameter less than $\delta$ are necessarily $\eps$-stable/unstable continua (see \cite{ArD}*{Proposition 2.3.1}). 
For $\sigma=s$, note that each set $f^{-n}(C^s_{\eps}(f^n(x)))$
is an arc and that $(\bigcup_{n=1}^k f^{-n}(C^s_{\eps}(f^n(x))))_{k\in\N}$ is a nested sequence of arcs. For $\sigma=u$, use the nested arcs $f^n(C^u_{\eps}(f^{-n}(x)))$ instead.
%since if $z\in f^{-n}(C^\sigma_{\eps}(f^n(x)))\cap f^{-m}(C^\sigma_{\eps}(f^m(x)))$
%for $m\neq n$ then by the above results there is an open set $z\in V$ such that 
%$f^{-n}(C^\sigma_{\eps}(f^n(x)))\cap f^{-m}(C^\sigma_{\eps}(f^m(x)))\cap V$ is an arc. Therefore
This implies that $C^\sigma(x)$ is a homeomorphic image of either $\mathbb{R}$,
$[0,\infty)$, or $[0,1]$. The case $[0,1]$ is impossible, because we have finitely many spines, and as a result $C^\sigma(x)$ would be periodic (as its endpoints would be periodic spines), which contradicts the forward (respectively backward) asymptoticity of its two distinct periodic endpoints. This proves \eqref{fol:2}, and to prove \eqref{fol:3} we note that the transversality between any pair of local stable/unstable continua was proved in \cite{ACS}*{Lemma 2.16}. Thus, any intersection between a pair of stable/unstable leaves can be seen as an intersection between small pieces of these leaves, which are local stable/unstable continua and, consequently, intersect transversally.

Next we are going to prove \eqref{fol:4}. % as a consequence of the . 
By Definition~\ref{singfol} for any $x\in S\setminus \spine(f)$ and any two charts $(U,\psi)$, $(V,\phi)$ with $U,V\subset U_x$ we can find a transition chart $$\eta=\psi\circ \phi^{-1}\colon \phi(U\cap V)\to \psi (U\cap V)$$ of the form $$\eta(y,z)=(\alpha(y,z),\gamma(z)).$$
Therefore, we can introduce an equivalence relation on charts at $x$: $(U,\psi)$ and $(V,\phi)$ are related provided the transverse component $\gamma$ is orientation preserving. For each $x\in S\setminus \spine(f)$ we have exactly two equivalence classes. Following \cite[p.17]{HeHi} denote by $S^*$
the space of all equivalence classes of charts over all points $x\in S\setminus \spine(f)$.
For any open set $U\subset U_x$ we obtain two sets consisting of these classes in $S^*$
that have either the same orientation as $\phi_x$ or the opposite orientation. Defining such sets for any $U$ as above and any $x\in S\setminus \spine(f)$ we obtain a base for the topology for $S^*$. In this way, we obtain a continuous 2-to-1 covering map $$\pi \colon S^*\to S\setminus \spine(f)$$ and this ensures that $S^*$ is a surface. Note that since $S\setminus \spine(f)$ is arcwise connected, there are at most two (arcwise) connected components in $S^*$. 

Assume first that $\spine(f)\neq\emptyset$.
We claim that
$S^*$ is connected.  Fix $p\in\spine(f)$ and let $\gamma\colon[0,1]\to S$ be a small simple loop in $S\setminus\spine(f)$ encircling a small disc containing
$p$ and no other spine. By \eqref{fol:1} a neighborhood of $p$ is modeled by the $1$-prong chart
$\mathcal{D}_1=\pi_2(\mathcal{D}_2)$, on which the two local transverse-orientation classes of $\mc{F}^\sigma_f$ are
interchanged by the monodromy of the branched cover $\pi_2$ along $\gamma$. Hence the
orientation class is reversed after one revolution around $\gamma$, so $\gamma$ does not lift
to a loop in $S^*$ but to a path joining the two points of $\pi^{-1}(\gamma(0))$. Therefore, the two sheets lie in a single arcwise-connected component, and so $S^*$ is connected.

Since points $p\in \spine(f)$ are isolated, we may compactify $S^*$ by adding points $p^*$ such that $p^*=\underset{n\to\infty}{\lim} x_n$ with $(x_n)_{n\in\N}\subset S^*$ satisfying $\underset{n\to\infty}{\lim} \pi(x_n)=p$. Write 
$$\tilde S=S^*\cup \{p^* : p\in \spine(f)\}$$ and extend $\pi$ to $\tilde S$ by putting $\pi(p^*)=p$. At each added point use the local model $z\mapsto z^2$ over the one-prong chart and note that $\tilde S$ is a surface and $\pi$ is a branched double covering. Clearly, $\pi\colon\tilde S\to S$ remains continuous. 

Summing up, for each regular point $x\in S$, there is a neighborhood 
$U_x$ of $x$ such that its pre-image $\pi^{-1}(U_x)$ is the union of two disjoint open sets in $\widetilde{S}$ that project homeomorphically onto $U_x$, while for each $p\in\spine(f)$, there is a neighborhood $V_p$ of $p$ such that its pre-image $\pi^{-1}(V_p)$ is an open set in $\widetilde{S}$ where $\pi$ is 2-to-1, except obviously at $p$ which satisfies $\#\pi^{-1}(p)=1$. To prove that $\tilde S$ is compact, we will prove that each sequence $(z_n)_{n\in\N}\subset\tilde{S}$ has a converging subsequence. Its projection $(x_n)_{n\in\N}$ defined by $x_n = \pi(z_n)$ has a subsequence $(x_{n_k})_{k\in\N}$ converging to some $p\in S$. If $p$ is regular, then $p$ has a neighborhood $U$ such that $\pi^{-1}(U)$ is a disjoint union of two open sets homeomorphic to $U$. Thus, a subsequence of $(z_{n_k})_{k\in\N}$ must converge to a point in one of these sheets. If $p \in \spine(f)$, the definition of the added point $p^*$ ensures that $(z_{n_k})$ converges to $p^*$ in $\tilde{S}$.

The pre-image of each stable (unstable) leaf in $\mc{F}^s_f$ ($\mc{F}^u_f$) that does not contain a spine is the union of two disjoint leaves of the lifted foliation homeomorphic to $\R$, while the pre-image of a leaf that contains a spine is homeomorphic to $\mathbb{R}$. Thus, the lifted foliations do not contain spines and satisfy a local product structure (as in \cite{APV}*{Definition 3.10}). 

Since $f$ preserves the foliation, transporting transverse orientations defines $\tilde f(x,o)=(f(x),f_*o)$ on $S^*$.
But $\spine(f)$ is a finite set consisting of periodic orbits of $f$ and $\pi$ is 1-to-1 onto $\spine(f)$, hence we obtain a well-defined lift $\widetilde{f}\colon\widetilde{S}\to\widetilde{S}$ of
$f$
such that its foliations $\mc{F}^s_{\widetilde{f}}$ and $\mc{F}^u_{\widetilde{f}}$ coincide with the lifted foliations of $f$ by $\pi$. 

Compactness and finiteness of the fibers give, for every $\eta>0$, a $\delta>0$ such that
$$\diam(\pi(K))<\delta\quad\Longrightarrow\quad\diam (K)<\eta$$
for every continuum $K\subset\widetilde S$; otherwise a Hausdorff limit would be a nontrivial continuum in a fiber. Consequently, if a lifted stable arc $I$ projects into $C^s_{\frac{\delta}{3}}(\pi(\tilde x))$ and contains $\tilde x$, then
$$\diam(\pi(\tilde f^n(I)))\leq \frac{2\delta}{3}\quad(\forall \,\,n\geq0)\qquad\text{and}\qquad I\subset C^s_\eta(\tilde x).$$
The analogous statement holds for unstable arcs. Conversely, uniform continuity of $\pi$ places the projection of each sufficiently small local stable/unstable continuum in the corresponding local continuum of $f$, hence the continuum itself in a lifted stable/unstable arc. Thus, the regular lifted charts give the cw-local-product-structure for $\tilde f$ and
$$C^s_\eta(\tilde x)\cap C^u_\eta(\tilde x)=\{\tilde x\}$$
for a uniform sufficiently small $\eta>0$. Moreover, a continuum with sufficiently small iterates projects to a point by cw-expansiveness of $f$, and is therefore a point. Hence $\tilde f$ is cw$_1$-hyperbolic 
and, consequently, $\widetilde f$ is topologically hyperbolic by \cite[Theorem~3.1]{ACCV3}. The Hiraide/Lewowicz classification of topologically hyperbolic homeomorphisms on surfaces \cite{H} ensures that $\widetilde{f}$ is conjugate to an Anosov automorphism of $\mathbb{T}^2$. Thus, we can assume that $\tilde f$ itself is an Anosov automorphism and the factor map is $\pi$. Also, this proves, in particular, that $\mc{F}^s_{\widetilde{f}}$ and $\mc{F}^u_{\widetilde{f}}$ are minimal foliations, which ensure that their projections $\mc{F}^s_f$ and $\mc{F}^u_f$ are minimal foliations on $S$, proving (\ref{fol:4}) and concluding the proof for the case $\spine(f)\neq \emptyset$.

Finally, if $\spine(f)=\emptyset$, compactness and the regular charts give $C^s_\eta(x)\cap C^u_\eta(x)=\{x\}$ for a uniform $\eta>0$. Together with the assumed cw-local-product-structure, this gives cw$_1$-hyperbolicity, hence topological hyperbolicity by \cite[Theorem~3.1]{ACCV3}. The Hiraide/Lewowicz classification gives that $f$ is conjugate to an Anosov automorphism of $\T^2$, which also ensures that $\mc{F}^s_f$ and $\mc{F}^u_f$ are minimal and completes the proof.
\end{proof}

A direct consequence of Theorem \ref{finitespines}, Theorem \ref{foliation}, and \cite[Proposition B]{Hi2} is that $f$ is a pseudo-Anosov homeomorphism with all singularities being spines. For a nontrivial sufficiently short stable arc $I$, contraction and continuity of the transverse measure give $\lambda^{-n}\mu^u(I)=\mu^u(f^n(I))\to0$, so $\lambda>1$. This proves the first implication of Theorem \ref{Bcw}.

In what follows, we discuss the opposite direction, aiming at the proof that pseudo-Anosov surface homeomorphisms with singularities being spines are cw$_F$-hyperbolic.
Let us recall some important definitions from \cite{BC} and \cite{FLPKM}. 

\begin{definition}[Transverse invariant measure] Let $\mc{F}$ be a foliation of $S$ with isolated singularities. By a \textit{transverse invariant measure} we mean a non-atomic, finite Borel measure $\mu$ that is defined on each arc transverse to the foliation, is positive on non-trivial arcs, and satisfies the following invariance property: if $\alpha,\beta\colon [0, 1] \to M$ are two arcs that are transverse to $\mc{F}$ and
are isotopic through transverse arcs whose endpoints remain in the
same leaf, then $\mu(\alpha([0,1])) = \mu(\beta([0,1]))$. If an arc passes through a singularity, the transversality pertains to all points of the arc belonging to a regular leaf, so we can also measure these arcs.
\end{definition}

\begin{definition}[Transverse measured foliation]
If $\mc{F}$ is a $C^0$ singular foliation and $\mu$ is a transverse invariant measure for $\mc{F}$, then the pair $(\mc{F},\mu)$ is called a \textit{transverse measured $C^0$ singular foliation}.
\end{definition}

The following definition of a pseudo-Anosov surface homeomorphism with spines generalizes the definition given in \cite{Hi2} of a pseudo-Anosov homeomorphism. The reader is referred to \cite[Section 13]{FLPKM} for further information on pseudo-Anosov homeomorphisms. In general, pseudo-Anosov homeomorphisms are defined assuming that $p(x)\geq2$ holds for every $x\in S$. This is necessary in the case where $f$ is expansive. In our definition, we rule out this assumption while keeping all other features of pseudo-Anosov homeomorphisms. This, in particular, allows for the existence of spines and includes the case of cw$_F$-hyperbolic homeomorphisms.

\begin{definition}[Pseudo-Anosov homeomorphism with spines] Let $S$ be a closed surface and $d$ be a Riemannian metric on it. We say that a homeomorphism $f\colon S\to S$ is \emph{pseudo-Anosov with spines} if there exist a constant $\lambda>1$, called the dilatation factor, and a pair of transverse measured $C^0$ singular foliations $(\mc{F}^s,\mu^s)$ and $(\mc{F}^u,\mu^u)$ with a finite number of singularities, such that $\mc{F}^s$ is transverse to $\mc{F}^u$, $$f(\mc{F}^s,\mu^s)=(\mc{F}^s,\lambda^{-1}\mu^s) \quad \text{and} \quad f(\mc{F}^u,\mu^u)=(\mc{F}^u,\lambda\mu^u),$$ 
where $f(\mc{F,\mu})
= (\{f(L): L\in \mc{F}\},f_*\mu)$ and $f_*\mu$ is the push-forward measure.
\end{definition}

\begin{remark}
We do not require the invariant foliations of a pseudo-Anosov
homeomorphism with spines to be minimal. For a cw$_F$-hyperbolic surface homeomorphism $f$ this is not a loss:
minimality of $\mc{F}^s_f$ and $\mc{F}^u_f$ is part of the conclusion of Theorem~\ref{foliation} \eqref{fol:4}. Thus, on
the cw$_F$-hyperbolic side, our definition agrees with the classical one wherever the latter applies.
\end{remark}

Roughly speaking, being pseudo-Anosov means that $f$ preserves the singular foliations $\mc{F}^s$ and $\mc{F}^u$, contracts all the arcs on the leaves of $\mc{F}^s$ by $\lambda^{-1}$, and expands all the arcs on the leaves of $\mc{F}^u$ by $\lambda$. These contractions/expansions are seen using the measures $\mu^u$ and $\mu^s$ as in the above equalities. When all singularities are spines, we can use these measures to define a metric $d_{\rho}$ on $S$ as follows (this is also based on Section 4 of \cite{C}). 

For any continuous path $\gamma \colon [0,1] \to S$, we evaluate its length by tracking its total variation with respect to the transverse measures. Given a sufficiently fine partition $P = \{0 = t_0, t_1, \ldots, t_m = 1\}$, each segment $\gamma([t_i, t_{i+1}])$ is contained entirely within a local regular or singular chart. Within this chart, we define $\mu^u(\gamma([t_i, t_{i+1}]))$ as the $\mu^u$-measure of any transverse path connecting the local leaf of $\mathcal{F}^s$ containing $\gamma(t_i)$ to the local leaf of $\mathcal{F}^s$ containing $\gamma(t_{i+1})$. Symmetrically, $\mu^s(\gamma([t_i, t_{i+1}]))$ denotes the $\mu^s$-measure of the transverse displacement between the local leaves of $\mathcal{F}^u$ containing the segment's endpoints. The length of $\gamma$ is then defined as
$$\ell_\rho(\gamma) = \sup_P \sum_{i=0}^{m-1} \sqrt{[\mu^u(\gamma([t_i, t_{i+1}]))]^2 + [\mu^s(\gamma([t_i, t_{i+1}]))]^2},$$
where the supremum is taken over all such partitions of $[0,1]$. We say $\gamma$ is \textit{rectifiable} if $\ell_\rho(\gamma) < +\infty$. Denoting the set of all rectifiable paths in $S$ by $\mathcal{R}$, we define the distance function $d_\rho \colon S \times S \to \mathbb{R}^+$ by
$$d_\rho(x,y) = \inf \{\ell_\rho(\gamma) \mid \gamma \in \mathcal{R}, \gamma(0)=x, \text{ and } \gamma(1)=y \}.$$
Because constant maps trivially yield zero transverse displacement, it follows naturally that $d_\rho(x,x) = 0$.

The following Lemma is inspired by the proofs for pseudo-Anosov diffeomorphisms (see \cite{FLPKM}, Section 9). Since we were unable to find a precise statement suitable for our setting, we present the proof for completeness.

\begin{lemma}\label{lem:drho}
The function $d_\rho$ is a metric uniformly equivalent to $d$. 
Furthermore, for each $x\in S\setminus \spine(f)$, there is an open set $U_x$ such that if an arc
$I\subset \mc{F}^s\cap U_x$ $($resp. $I\subset \mc{F}^u\cap U_x)$ 
then $\diam_{d_\rho}(I)= \mu^u(I)$ $($resp. $\diam_{d_\rho}(I)=\mu^s(I))$.
\end{lemma}

\begin{proof}
It is clear that $d_\rho$ satisfies the triangle inequality since the concatenation of two rectifiable paths is a rectifiable path itself. Also, since $S$ is covered by finitely many charts modeled on $\mc{D}_1$ and $\mc{D}_2$, any path connecting distinct points $x,y$ must traverse a strictly positive amount of transverse measure in at least one chart. This implies $d_\rho(x,y)>0$ for $x\neq y$. We can use topological rectangles in $\mc{D}_1$ or $\mc{D}_2$ to show that any open set in the topology of $d$ contains a ball in $d_\rho$ and vice versa (see Figure~\ref{fig:sketch}). Since $S$ is compact, $d$ and $d_\rho$ are uniformly equivalent.

\begin{figure}[ht]
\centering
\resizebox{0.40\linewidth}{!}{%
\begin{tikzpicture}[x=0.24pt,y=0.24pt,yscale=-1,xscale=1]
\useasboundingbox (53,0) rectangle (583,515);

\draw [color={rgb, 255:red, 30; green, 0; blue, 255 }  ,draw opacity=1 ]   (170.76,14.37) .. controls (120.93,203.73) and (120.93,256.05) .. (145.85,418) ;
\draw [color={rgb, 255:red, 30; green, 0; blue, 255 }  ,draw opacity=1 ]   (220.59,14.37) .. controls (170.76,203.73) and (170.76,256.05) .. (195.68,418) ;
\draw [color={rgb, 255:red, 30; green, 0; blue, 255 }  ,draw opacity=1 ]   (270.42,14.37) .. controls (220.59,203.73) and (220.59,256.05) .. (245.51,418) ;
\draw [color={rgb, 255:red, 30; green, 0; blue, 255 }  ,draw opacity=1 ]   (320.25,14.37) .. controls (270.42,203.73) and (270.42,256.05) .. (295.34,418) ;
\draw [color={rgb, 255:red, 30; green, 0; blue, 255 }  ,draw opacity=1 ]   (370.09,14.37) .. controls (320.25,203.73) and (320.25,256.05) .. (345.17,418) ;
\draw [color={rgb, 255:red, 30; green, 0; blue, 255 }  ,draw opacity=1 ]   (419.92,14.37) .. controls (370.09,203.73) and (370.09,256.05) .. (395,418) ;
\draw [color={rgb, 255:red, 30; green, 0; blue, 255 }  ,draw opacity=1 ]   (469.75,14.37) .. controls (419.92,203.73) and (419.92,256.05) .. (444.83,418) ;
\draw [color={rgb, 255:red, 30; green, 0; blue, 255 }  ,draw opacity=1 ]   (519.58,14.37) .. controls (469.75,203.73) and (469.75,256.05) .. (494.66,418) ;
\draw [color={rgb, 255:red, 255; green, 5; blue, 36 }  ,draw opacity=1 ]   (101,49.25) .. controls (367.59,11.88) and (340.19,39.29) .. (556.95,54.24) ;
\draw [color={rgb, 255:red, 255; green, 5; blue, 36 }  ,draw opacity=1 ]   (101,99.08) .. controls (367.59,61.71) and (340.19,89.12) .. (556.95,104.07) ;
\draw [color={rgb, 255:red, 255; green, 5; blue, 36 }  ,draw opacity=1 ]   (101,148.91) .. controls (367.59,111.54) and (340.19,138.95) .. (556.95,153.9) ;
\draw [color={rgb, 255:red, 255; green, 5; blue, 36 }  ,draw opacity=1 ]   (101,198.75) .. controls (367.59,161.37) and (340.19,188.78) .. (556.95,203.73) ;
\draw [color={rgb, 255:red, 255; green, 5; blue, 36 }  ,draw opacity=1 ]   (101,248.58) .. controls (367.59,211.2) and (340.19,238.61) .. (556.95,253.56) ;
\draw [color={rgb, 255:red, 255; green, 5; blue, 36 }  ,draw opacity=1 ]   (101,298.41) .. controls (367.59,261.03) and (340.19,288.44) .. (556.95,303.39) ;
\draw [color={rgb, 255:red, 255; green, 5; blue, 36 }  ,draw opacity=1 ]   (101,348.24) .. controls (367.59,310.86) and (340.19,338.27) .. (556.95,353.22) ;
\draw [color={rgb, 255:red, 255; green, 5; blue, 36 }  ,draw opacity=1 ]   (101,398.07) .. controls (367.59,360.69) and (340.19,388.1) .. (556.95,403.05) ;
\draw [color={rgb, 255:red, 30; green, 0; blue, 255 }  ,draw opacity=1 ]   (569.41,14.37) .. controls (519.58,203.73) and (519.58,256.05) .. (544.49,418) ;
\draw  [fill={rgb, 255:red, 0; green, 0; blue, 0 }  ,fill opacity=1 ] (321.15,207.75) .. controls (321.15,205.26) and (319.13,203.24) .. (316.64,203.24) .. controls (314.14,203.24) and (312.12,205.26) .. (312.12,207.75) .. controls (312.12,210.25) and (314.14,212.27) .. (316.64,212.27) .. controls (319.13,212.27) and (321.15,210.25) .. (321.15,207.75) -- cycle ;
\draw [color={rgb, 255:red, 0; green, 0; blue, 0 }  ,draw opacity=1 ][line width=2.25]    (256.12,138.21) .. controls (236.25,194.2) and (243.48,235.75) .. (245.29,291.75) ;
\draw [color={rgb, 255:red, 0; green, 0; blue, 0 }  ,draw opacity=1 ][line width=2.25]    (400.63,140.01) .. controls (380.76,196.01) and (387.99,237.56) .. (389.79,293.55) ;
\draw [color={rgb, 255:red, 0; green, 0; blue, 0 }  ,draw opacity=1 ][line width=2.25]    (256.12,138.21) .. controls (322.96,121.95) and (348.25,132.79) .. (400.63,140.01) ;
\draw [color={rgb, 255:red, 0; green, 0; blue, 0 }  ,draw opacity=1 ][line width=2.25]    (245.29,291.75) .. controls (312.12,275.49) and (337.41,286.33) .. (389.79,293.55) ;
\draw    (325.71,135.98) -- (316.79,196.02) ;
\draw [shift={(316.5,198)}, rotate = 278.44] [color={rgb, 255:red, 0; green, 0; blue, 0 }  ][line width=0.75]    (10.93,-3.29) .. controls (6.95,-1.4) and (3.31,-0.3) .. (0,0) .. controls (3.31,0.3) and (6.95,1.4) .. (10.93,3.29)   ;
\draw [shift={(326,134)}, rotate = 98.44] [color={rgb, 255:red, 0; green, 0; blue, 0 }  ][line width=0.75]    (10.93,-3.29) .. controls (6.95,-1.4) and (3.31,-0.3) .. (0,0) .. controls (3.31,0.3) and (6.95,1.4) .. (10.93,3.29)   ;
\draw    (248,209.97) -- (304,209.03) ;
\draw [shift={(306,209)}, rotate = 179.05] [color={rgb, 255:red, 0; green, 0; blue, 0 }  ][line width=0.75]    (10.93,-3.29) .. controls (6.95,-1.4) and (3.31,-0.3) .. (0,0) .. controls (3.31,0.3) and (6.95,1.4) .. (10.93,3.29)   ;
\draw [shift={(246,210)}, rotate = 359.05] [color={rgb, 255:red, 0; green, 0; blue, 0 }  ][line width=0.75]    (10.93,-3.29) .. controls (6.95,-1.4) and (3.31,-0.3) .. (0,0) .. controls (3.31,0.3) and (6.95,1.4) .. (10.93,3.29)   ;
\draw  [color={rgb, 255:red, 248; green, 120; blue, 28 }  ,draw opacity=1 ][line width=2.25]  (174.26,207.75) .. controls (174.26,129.12) and (238,65.38) .. (316.64,65.38) .. controls (395.27,65.38) and (459.01,129.12) .. (459.01,207.75) .. controls (459.01,286.38) and (395.27,350.13) .. (316.64,350.13) .. controls (238,350.13) and (174.26,286.38) .. (174.26,207.75) -- cycle ;
\draw [color={rgb, 255:red, 0; green, 0; blue, 0 }  ,draw opacity=1 ][line width=2.25]    (185.63,54.77) .. controls (141.66,178.69) and (157.65,270.63) .. (161.64,394.56) ;
\draw [color={rgb, 255:red, 0; green, 0; blue, 0 }  ,draw opacity=1 ][line width=2.25]    (505.43,58.77) .. controls (461.46,182.69) and (477.45,274.63) .. (481.44,398.55) ;
\draw [color={rgb, 255:red, 0; green, 0; blue, 0 }  ,draw opacity=1 ][line width=2.25]    (185.63,54.77) .. controls (333.54,18.79) and (389.5,42.78) .. (505.43,58.77) ;
\draw [color={rgb, 255:red, 0; green, 0; blue, 0 }  ,draw opacity=1 ][line width=2.25]    (161.64,394.56) .. controls (309.55,358.58) and (365.52,382.56) .. (481.44,398.55) ;
\draw    (156.87,61) -- (134.13,397) ;
\draw [shift={(134,399)}, rotate = 273.87] [color={rgb, 255:red, 0; green, 0; blue, 0 }  ][line width=0.75]    (10.93,-3.29) .. controls (6.95,-1.4) and (3.31,-0.3) .. (0,0) .. controls (3.31,0.3) and (6.95,1.4) .. (10.93,3.29)   ;
\draw [shift={(157,59)}, rotate = 93.87] [color={rgb, 255:red, 0; green, 0; blue, 0 }  ][line width=0.75]    (10.93,-3.29) .. controls (6.95,-1.4) and (3.31,-0.3) .. (0,0) .. controls (3.31,0.3) and (6.95,1.4) .. (10.93,3.29)   ;
\draw    (483,409) -- (154,409) ;
\draw [shift={(152,409)}, rotate = 360] [color={rgb, 255:red, 0; green, 0; blue, 0 }  ][line width=0.75]    (10.93,-3.29) .. controls (6.95,-1.4) and (3.31,-0.3) .. (0,0) .. controls (3.31,0.3) and (6.95,1.4) .. (10.93,3.29)   ;
\draw [shift={(485,409)}, rotate = 180] [color={rgb, 255:red, 0; green, 0; blue, 0 }  ][line width=0.75]    (10.93,-3.29) .. controls (6.95,-1.4) and (3.31,-0.3) .. (0,0) .. controls (3.31,0.3) and (6.95,1.4) .. (10.93,3.29)   ;
\draw    (470,440) .. controls (508.81,419.1) and (556.52,351.68) .. (486.07,316.53) ;
\draw [shift={(485,316)}, rotate = 25.92] [color={rgb, 255:red, 0; green, 0; blue, 0 }  ][line width=0.75]    (10.93,-3.29) .. controls (6.95,-1.4) and (3.31,-0.3) .. (0,0) .. controls (3.31,0.3) and (6.95,1.4) .. (10.93,3.29)   ;
\draw [color={rgb, 255:red, 248; green, 120; blue, 28 }  ,draw opacity=1 ]   (156,448) .. controls (195.6,418.3) and (164.63,340.58) .. (202.82,309.91) ;
\draw [shift={(204,309)}, rotate = 143.13] [color={rgb, 255:red, 248; green, 120; blue, 28 }  ,draw opacity=1 ][line width=0.75]    (10.93,-3.29) .. controls (6.95,-1.4) and (3.31,-0.3) .. (0,0) .. controls (3.31,0.3) and (6.95,1.4) .. (10.93,3.29)   ;

\draw (318.64,206.64) node [anchor=north west][inner sep=0.75pt]  [font=\large,xscale=0.45,yscale=0.45]  {$x$};
\draw (324,156.4) node [anchor=north west][inner sep=0.75pt]  [font=\large,xscale=0.45,yscale=0.45]  {$\mu ^{s}  >\delta $};
\draw (248,213.4) node [anchor=north west][inner sep=0.75pt]  [font=\large,xscale=0.45,yscale=0.45]  {$\mu ^{u}  >\delta $};
\draw (59,199.4) node [anchor=north west][inner sep=0.75pt]  [font=\large,xscale=0.45,yscale=0.45]  {$\mu ^{s} < \frac{\varepsilon }{2}$};
\draw (275,410.4) node [anchor=north west][inner sep=0.75pt]  [font=\large,xscale=0.45,yscale=0.45]  {$\mu ^{u} < \frac{\varepsilon }{2}$};
\draw (398,444) node [anchor=north west][inner sep=0.75pt]  [xscale=0.45,yscale=0.45] [align=left] {By triangle inequality\\this square is in};
\draw (397,486.4) node [anchor=north west][inner sep=0.75pt]  [xscale=0.45,yscale=0.45]  {$B_{d_{\rho}}( x,\varepsilon )$};
\draw (79,452) node [anchor=north west][inner sep=0.75pt]  [xscale=0.45,yscale=0.45] [align=left] {By triangle inequality\\this contains};
\draw (187,472.4) node [anchor=north west][inner sep=0.75pt]  [xscale=0.45,yscale=0.45]  {$B_{d_{\rho}}( x,\delta )$};
\draw (398,66.4) node [anchor=north west][inner sep=0.75pt]  [font=\large,color={rgb, 255:red, 248; green, 120; blue, 28 }  ,opacity=1 ,xscale=0.45,yscale=0.45]  {$B( x,\gamma )$};

\end{tikzpicture}
}
\hfill
\resizebox{0.49\linewidth}{!}{%
\begin{tikzpicture}[x=0.24pt,y=0.24pt,yscale=-1,xscale=1]
\useasboundingbox (40,0) rectangle (640,480);
\draw [color={rgb, 255:red, 30; green, 0; blue, 255 }  ,draw opacity=1 ]   (52,197.06) -- (327.79,197.06) ;
\draw [color={rgb, 255:red, 255; green, 5; blue, 36 }  ,draw opacity=1 ]   (327.79,197.06) -- (603.58,197.06) ;
\draw  [fill={rgb, 255:red, 0; green, 0; blue, 0 }  ,fill opacity=1 ] (332.6,197.06) .. controls (332.6,194.41) and (330.45,192.25) .. (327.79,192.25) .. controls (325.13,192.25) and (322.98,194.41) .. (322.98,197.06) .. controls (322.98,199.72) and (325.13,201.87) .. (327.79,201.87) .. controls (330.45,201.87) and (332.6,199.72) .. (332.6,197.06) -- cycle ;
\draw [color={rgb, 255:red, 255; green, 5; blue, 36 }  ,draw opacity=1 ]   (603.58,233.94) .. controls (63.22,243.56) and (72.84,150.56) .. (603.58,153.77) ;
\draw [color={rgb, 255:red, 255; green, 5; blue, 36 }  ,draw opacity=1 ]   (600.37,266.01) .. controls (-10.53,272.42) and (-13.74,116.89) .. (601.98,123.3) ;
\draw [color={rgb, 255:red, 255; green, 5; blue, 36 }  ,draw opacity=1 ]   (600.37,298.08) .. controls (-79.48,302.89) and (-81.08,83.22) .. (606.79,91.24) ;
\draw [color={rgb, 255:red, 255; green, 5; blue, 36 }  ,draw opacity=1 ]   (597.17,342.97) .. controls (-138.81,346.18) and (-117.96,46.34) .. (608.39,59.17) ;
\draw [color={rgb, 255:red, 30; green, 0; blue, 255 }  ,draw opacity=1 ]   (59.15,159.81) .. controls (599.5,150.19) and (589.88,243.19) .. (59.15,239.98) ;
\draw [color={rgb, 255:red, 30; green, 0; blue, 255 }  ,draw opacity=1 ]   (62.36,127.74) .. controls (673.26,121.33) and (676.47,276.86) .. (60.75,270.45) ;
\draw [color={rgb, 255:red, 30; green, 0; blue, 255 }  ,draw opacity=1 ]   (62.36,95.67) .. controls (742.21,90.86) and (743.81,310.53) .. (55.94,302.51) ;
\draw [color={rgb, 255:red, 30; green, 0; blue, 255 }  ,draw opacity=1 ]   (65.56,50.78) .. controls (801.54,47.57) and (780.69,347.41) .. (54.34,334.58) ;
\draw [line width=2.25]    (338,230) .. controls (348,227) and (47,205) .. (313,164) ;
\draw [line width=2.25]    (338,230) .. controls (399,227) and (593,189) .. (313,164) ;
\draw    (242.99,186.18) -- (315.01,192.82) ;
\draw [shift={(317,193)}, rotate = 185.26] [color={rgb, 255:red, 0; green, 0; blue, 0 }  ][line width=0.75]    (10.93,-3.29) .. controls (6.95,-1.4) and (3.31,-0.3) .. (0,0) .. controls (3.31,0.3) and (6.95,1.4) .. (10.93,3.29)   ;
\draw [shift={(241,186)}, rotate = 5.26] [color={rgb, 255:red, 0; green, 0; blue, 0 }  ][line width=0.75]    (10.93,-3.29) .. controls (6.95,-1.4) and (3.31,-0.3) .. (0,0) .. controls (3.31,0.3) and (6.95,1.4) .. (10.93,3.29)   ;
\draw    (341.99,203.18) -- (414.01,209.82) ;
\draw [shift={(416,210)}, rotate = 185.26] [color={rgb, 255:red, 0; green, 0; blue, 0 }  ][line width=0.75]    (10.93,-3.29) .. controls (6.95,-1.4) and (3.31,-0.3) .. (0,0) .. controls (3.31,0.3) and (6.95,1.4) .. (10.93,3.29)   ;
\draw [shift={(340,203)}, rotate = 5.26] [color={rgb, 255:red, 0; green, 0; blue, 0 }  ][line width=0.75]    (10.93,-3.29) .. controls (6.95,-1.4) and (3.31,-0.3) .. (0,0) .. controls (3.31,0.3) and (6.95,1.4) .. (10.93,3.29)   ;
\draw  [color={rgb, 255:red, 248; green, 120; blue, 28 }  ,draw opacity=1 ][line width=2.25]  (183.29,197.06) .. controls (183.29,117.26) and (247.98,52.56) .. (327.79,52.56) .. controls (407.59,52.56) and (472.29,117.26) .. (472.29,197.06) .. controls (472.29,276.87) and (407.59,341.56) .. (327.79,341.56) .. controls (247.98,341.56) and (183.29,276.87) .. (183.29,197.06) -- cycle ;
\draw [line width=2.25]    (147,27) -- (146,366) ;
\draw [line width=2.25]    (509,366) -- (146,366) ;
\draw [line width=2.25]    (510,27) -- (147,27) ;
\draw [line width=2.25]    (510,27) -- (509,366) ;
\draw    (155.49,356.66) -- (318.51,210.34) ;
\draw [shift={(320,209)}, rotate = 138.09] [color={rgb, 255:red, 0; green, 0; blue, 0 }  ][line width=0.75]    (10.93,-3.29) .. controls (6.95,-1.4) and (3.31,-0.3) .. (0,0) .. controls (3.31,0.3) and (6.95,1.4) .. (10.93,3.29)   ;
\draw [shift={(154,358)}, rotate = 318.09] [color={rgb, 255:red, 0; green, 0; blue, 0 }  ][line width=0.75]    (10.93,-3.29) .. controls (6.95,-1.4) and (3.31,-0.3) .. (0,0) .. controls (3.31,0.3) and (6.95,1.4) .. (10.93,3.29)   ;
\draw    (134,364) -- (134,32) ;
\draw [shift={(134,30)}, rotate = 90] [color={rgb, 255:red, 0; green, 0; blue, 0 }  ][line width=0.75]    (10.93,-3.29) .. controls (6.95,-1.4) and (3.31,-0.3) .. (0,0) .. controls (3.31,0.3) and (6.95,1.4) .. (10.93,3.29)   ;
\draw [shift={(134,366)}, rotate = 270] [color={rgb, 255:red, 0; green, 0; blue, 0 }  ][line width=0.75]    (10.93,-3.29) .. controls (6.95,-1.4) and (3.31,-0.3) .. (0,0) .. controls (3.31,0.3) and (6.95,1.4) .. (10.93,3.29)   ;
\draw    (172,416) .. controls (135.74,407.18) and (147.5,392.6) .. (169.63,376.02) ;
\draw [shift={(171,375)}, rotate = 143.53] [color={rgb, 255:red, 0; green, 0; blue, 0 }  ][line width=0.75]    (10.93,-3.29) .. controls (6.95,-1.4) and (3.31,-0.3) .. (0,0) .. controls (3.31,0.3) and (6.95,1.4) .. (10.93,3.29)   ;
\draw [color={rgb, 255:red, 248; green, 120; blue, 28 }  ,draw opacity=1 ]   (464,399) .. controls (426.57,393.09) and (429.89,360.01) .. (419.49,328.44) ;
\draw [shift={(419,327)}, rotate = 71.03] [color={rgb, 255:red, 248; green, 120; blue, 28 }  ,draw opacity=1 ][line width=0.75]    (10.93,-3.29) .. controls (6.95,-1.4) and (3.31,-0.3) .. (0,0) .. controls (3.31,0.3) and (6.95,1.4) .. (10.93,3.29)   ;

\draw (269,170.4) node [anchor=north west][inner sep=0.75pt]  [xscale=0.45,yscale=0.45]  {$\mu ^{s} =\delta $};
\draw (378,184.4) node [anchor=north west][inner sep=0.75pt]  [xscale=0.45,yscale=0.45]  {$\mu ^{u} =\delta $};
\draw (247,266.4) node [anchor=north west][inner sep=0.75pt]  [font=\large,xscale=0.45,yscale=0.45]  {$\mu ^{s} < \frac{\varepsilon }{2}$};
\draw (66,198.4) node [anchor=north west][inner sep=0.75pt]  [font=\large,xscale=0.45,yscale=0.45]  {$\mu ^{u} < \frac{\varepsilon }{2}$};
\draw (175,398) node [anchor=north west][inner sep=0.75pt]  [xscale=0.45,yscale=0.45] [align=left] {This square is contained\\in};
\draw (194,419.4) node [anchor=north west][inner sep=0.75pt]  [xscale=0.45,yscale=0.45]  {$B_{d_{\rho}}( x,\varepsilon )$};
\draw (470,391) node [anchor=north west][inner sep=0.75pt]  [xscale=0.45,yscale=0.45] [align=left] {Contains};
\draw (546,389.4) node [anchor=north west][inner sep=0.75pt]  [xscale=0.45,yscale=0.45]  {$B_{d_{\rho}}( x,\delta )$};

\end{tikzpicture}
}

    \caption{Sketch of the proof that balls of $d$ and $d_\rho$ can be included in each other in $\mc{D}_2$ and $\mc{D}_1$.}
    \label{fig:sketch}
\end{figure}
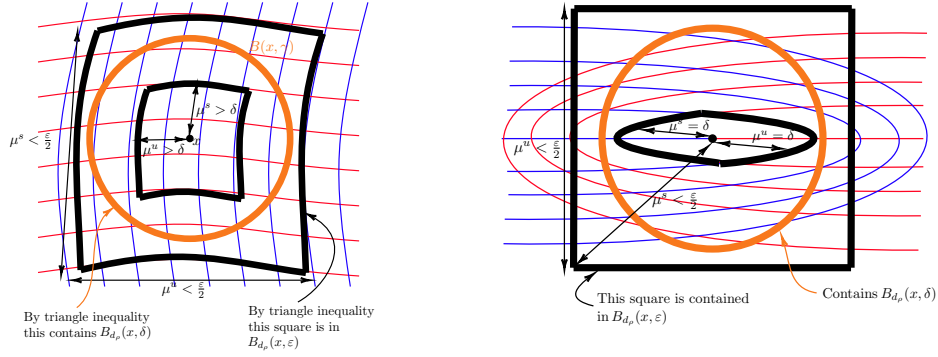

Let $z\in S\setminus \spine(f)$ be a regular point. By the definition of a $C^0$ singular foliation, there is a canonical neighborhood $U$ of $z$ equipped with a chart homeomorphic to $\mc{D}_2$. To prevent paths from trivially minimizing distance by exiting and re-entering the chart, we restrict to a smaller open neighborhood $V$ such that $z \in V \subset \overline{V} \subset U$. Because $\overline{V}$ is compact and strictly disjoint from the boundary $\partial U$, there exists a constant $\eta > 0$ such that any path originating in $\overline{V}$ and reaching $\partial U$ must accumulate an $\ell_\rho$-length of at least $\eta$. 
%We define $U_z$ as an open neighborhood of $z$ strictly inside $V$ such that for any $x, y \in U_z$, we have $\mu^u(s_{x,y}) < \eta$ and $\mu^s(u_{x,y}) < \eta$.
We define $U_z$ as an open neighborhood of $z$ strictly inside $V$ such that every stable arc $s_{x,y}\subset U_z$ and unstable arc $u_{x,y}\subset U_z$ satisfies $\mu^u(s_{x,y}) < \eta$ and $\mu^s(u_{x,y}) < \eta$.

We now prove that if $s_{x,y} \subset U_z$ is a stable arc connecting $x$ to $y$, then $d_\rho(x,y) = \mu^u(s_{x,y})$. Let $\gamma \colon [0,1] \to S$ be an arbitrary rectifiable path connecting $x$ to $y$. We consider two cases for $\gamma$: first, if $\gamma$ leaves $U$, it must cross from $V$ to $\partial U$, so its length is bounded below by the boundary distance, giving $\ell_\rho(\gamma) \ge \eta > \mu^u(s_{x,y})$; second, if $\gamma$ stays entirely within $U$, we compute its length using the local chart coordinates, that is, for any partition $P=\{0=t_0, t_1, \dots, t_m=1\}$, we have
\begin{eqnarray*}
\sum_{i=0}^{m-1} \sqrt{\left[ \mu^u(\gamma([t_i, t_{i+1}])) \right]^2 + \left[ \mu^s(\gamma([t_i, t_{i+1}])) \right]^2} &\geq& \sum_{i=0}^{m-1} \mu^u(\gamma([t_i, t_{i+1}]))\\
&\geq& \mu^u(\gamma([0,1])) \geq \mu^u(s_{x,y}).
\end{eqnarray*}
The final inequality holds because the projection of $\gamma$ along the $\mc{F}^s$ leaves inside the regular chart $\mc{D}_2$ must cover the entire transverse distance between the local stable leaf of $x$ and the local stable leaf of $y$. Taking the supremum over all partitions $P$, we obtain $\ell_{\rho}(\gamma)\geq \mu^u(s_{x,y})$. 
Taking the infimum over all paths $\gamma$ (both those leaving $U$ and those staying inside), we conclude $d_{\rho}(x,y)\geq \mu^u(s_{x,y})$.

To prove the reverse inequality, note that under our definition of $\mc{R}$, the stable arc $s_{x,y}$ itself is a valid rectifiable path. Because $s_{x,y}$ is strictly contained in a single leaf of $\mc{F}^s$, its transverse displacement with respect to $\mc{F}^s$ is everywhere zero, meaning $\mu^s(s_{x,y}([t_i, t_{i+1}])) = 0$ for any segment. Applying the length formula directly yields $\ell_\rho(s_{x,y}) = \mu^u(s_{x,y})$. Since $d_\rho$ is defined as the infimum over all rectifiable paths, it follows immediately that $d_\rho(x,y) \le \ell_\rho(s_{x,y}) = \mu^u(s_{x,y})$.
Together, this establishes $d_\rho(x,y) = \mu^u(s_{x,y})$. 

The case where $u_{x,y} \subset U_z$ is an unstable arc yielding $d_\rho(x,y) = \mu^s(u_{x,y})$ is completely symmetric.
\end{proof}

\begin{remark}
We note that on $\mc{D}_1$ we can get points $x$ and $y$ arbitrarily close to the stable arc of the spine and inside the same stable arc turning around the spine, so their distance $d_{\rho}(x,y)$ will not be given by $\mu^u(s_{x,y})$. Indeed, if $x$ and $y$ also belong to an unstable arc $u_{x,y}$ and $\mu^s(u_{x,y})<\mu^u(s_{x,y})$, it follows that $d_{\rho}(x,y)$ will be determined by $\mu^s(u_{x,y})$ and will increase when iterated. This means, in particular, that the metric $d_{\rho}$ is not hyperbolic. 
\end{remark}

Now we obtain an equivalence between all singularities of a pseudo-Anosov homeomorphism being spines and cw$_F$-hyperbolicity.

\begin{proof}[Proof of Theorem \ref{Bcw}]
The fact that cw$_F$-hyperbolicity implies $f$ is pseudo-Anosov with all singularities being spines was proved above. We now prove the converse implication. Assume that $f$ is a pseudo-Anosov homeomorphism with all singularities being spines. Let $(\mc{F}^s,\mu^s)$ and $(\mc{F}^u,\mu^u)$ denote its transverse measured stable/unstable $C^0$ singular foliations and $\lambda>1$ be its dilatation factor. We will prove that $f$ is cw$_2$-hyperbolic. First, we note that $\mc{F}^s/\mc{F}^u$ satisfy a local-product-structure since at regular points there is a local chart homeomorphic to $\mc{D}_2$ and at spines there is a local chart homeomorphic to $\mc{D}_1$ (see Definition~\ref{singfol}). In both cases, in a small neighborhood, each stable leaf intersects all unstable leaves and also each unstable leaf intersects all stable leaves.  More precisely, for each $\eps' > 0$, there exists a universal $\delta > 0$ such that if $d(x,y) < \delta$, then the local stable leaf segment of $x$ and the local unstable leaf segment of $y$ within these charts intersect at a point $z \in S$. Furthermore, by choosing $\delta$ sufficiently small, the sub-arcs $s_{x,z} \subset \mc{F}^s$ and $u_{y,z} \subset \mc{F}^u$ connecting $x$ to $z$ and $y$ to $z$, respectively, can be guaranteed to have small transverse measures:
$$z\in s_{x,z}\cap u_{y,z}, \,\,\, \mu^s(u_{y,z})<\eps', \,\,\, \text{and} \,\,\, \mu^u(s_{x,z})<\eps'.$$
Also note that pieces of stable/unstable leaves with small transverse measures are contained in local stable/unstable continua of $f$. Indeed, 
if $\mu^u(s_{x,z})<\eps'$, then
$$d_\rho(f^k(x),f^k(z))\leq\mu^u(f^k(s_{x,z}))=\lambda^{-k}\mu^u(s_{x,z})<\eps'$$
for every $k\in\N$, that is, $z\in C^s_{\eps'}(x)$ in the metric $d_{\rho}$. Also, if 
$\mu^s(u_{y,z})<\eps'$, then
$$d_\rho(f^{-k}(y),f^{-k}(z))\leq\mu^s(f^{-k}(u_{y,z}))=\lambda^{-k}\mu^s(u_{y,z})<\eps'$$
for every $k\in\N$, that is, $z\in C^u_{\eps'}(y)$ in the metric $d_{\rho}$. The same estimates hold for every point of the connecting arcs, so the whole arcs lie in the corresponding local stable/unstable continua.

Thus, for each $\eps>0$ we can choose $\eps'\in(0,\eps)$ given by the fact that $d$ and $d_{\rho}$ are uniformly equivalent metrics and $\delta\in(0,\eps')$ such that if $d(x,y)<\delta$, then there are $z\in S$ and arcs $s_{x,z}\subset\mc{F}^s$, $u_{y,z}\subset\mc{F}^u$ such that 
$$z\in s_{x,z}\cap u_{y,z}, \,\,\, \mu^s(u_{y,z})<\eps', \,\,\, \text{and} \,\,\, \mu^u(s_{x,z})<\eps'.$$ 
As noted above, we conclude that $z\in C^s_{\eps'}(x)\cap C^u_{\eps'}(y)$ in the metric $d_{\rho}$, which ensures that $z\in C^s_{\eps}(x)\cap C^u_{\eps}(y)$ in the metric $d$. This proves that $f$ satisfies the cw-local-product-structure.

Now we prove cw-expansiveness.  Let $\eps>0$ be small enough so that any ball of radius $2\eps$ in $S$ is contained within a canonical neighborhood homeomorphic to either a regular chart $\mc{D}_2$ or a $1$-prong singular chart $\mc{D}_1$. Because $f$ satisfies the cw-local-product-structure, associated with this $\eps$ there exists a uniform constant $\delta\in(0,\eps)$ ensuring that any two points closer than $\delta$ intersect via $\eps$-stable/unstable continua. Let $\{V_1, \dots, V_m\}$ be a finite open cover of $S$ by local charts such that the diameter of each $V_i$ is strictly less than $\delta$. Choose the singular charts $V_p$ so that $f(V_p)\cap V_q=\emptyset$ for all spines $p,q$ with $q\neq f(p)$. This sequential choice of scales ($\operatorname{diam}(V_i) < \delta < \eps$) guarantees that for any connected set contained entirely within a chart $V_i$, the local projections $\pi^s$ and $\pi^u$ along the stable and unstable leaves are rigorously well-defined, continuous, and do not cross the boundary of the parameterized domain.  Note that when $V_i$ is homeomorphic to $\mc{D}_1$, the projections on the stable/unstable regular leaves are not well-defined maps since each other regular leaf intersects them at two distinct points, but the projections on the stable/unstable leaves of the spine are well-defined since on $\mc{D}_1$ unstable/stable regular leaves intersect the stable/unstable leaves of the spine in a single point.

Since $S$ is compact, let $\alpha > 0$ be a Lebesgue number for this cover, meaning that any subset of $S$ with diameter less than $\alpha$ is entirely contained in at least one chart $V_i$. Since $d$ and $d_{\rho}$ are equivalent, there exists $\alpha'\in(0,\alpha)$ such that $d_{\rho}(x,y)\leq\alpha'$ implies $d(x,y)<\alpha$.
It follows directly from the definition that cw-expansiveness is invariant under equivalent metrics, so it is enough to prove cw-expansiveness using the metric $d_{\rho}$. 

Let $C\subset S$ be a non-trivial continuum satisfying \begin{equation}\label{diamcw}\operatorname{diam}_{d_\rho}(f^n(C)) \leq \alpha' \,\,\,\,\,\, \text{for every} \,\,\,\,\,\, n \in \mathbb{Z}.
\end{equation}
This ensures that $\diam(f^n(C))<\alpha$ and, consequently, $f^n(C)$ is contained in some  $V_{i(n)}$ for every $n\in\Z$. Thus, we can project $f^n(C)$ along the stable/unstable leaves onto the unstable/stable leaf of $x_{i(n)}$ inside $B(x_{i(n)},2\eps)$. For each $n\in\Z$, let $\mc{F}^u_{2\eps}(x_{i(n)})$ denote the connected component of $x_{i(n)}$ in $\mc{F}^u(x_{i(n)})\cap B(x_{i(n)},2\eps)$, $\mc{F}^s_{2\eps}(x_{i(n)})$ denote the connected component of $x_{i(n)}$ in $\mc{F}^s(x_{i(n)})\cap B(x_{i(n)},2\eps)$, and
$$\pi^s_n\colon f^n(C)\to\mc{F}^u_{2\eps}(x_{i(n)}), \,\,\,\, \pi^u_n\colon f^n(C)\to\mc{F}^s_{2\eps}(x_{i(n)})$$ be the projections along stable and unstable leaves.
Let $\pi^s=\pi^s_0$ and $\pi^u=\pi^u_0$.
Since $C$ is a  non-trivial continuum, at least one of its projections $\pi^s(C)$ or $\pi^u(C)$ must be a non-trivial continuum (a small arc) contained within a leaf. 

Assume that $\pi^s(C)$ is a non-trivial continuum, which ensures that $\mu^s(\pi^s(C))>0$. The metric $d_{\rho}$ inside the  $V_{i(n)}$ without spines satisfies: if $x'=\pi^s_n(x)$ and $y'=\pi^s_n(y)$ are the projections through the stable leaves of $x,y\in f^n(C)$, then $d_{\rho}(x,y)\geq\mu^s(u_{x',y'})$, where $u_{x',y'}$ is the unstable arc connecting $x'$ to $y'$ inside $\mc{F}^u(x_{i(n)})$, in particular transversal to $\mc{F}^s$ and so $\mu^s$ can be applied to it. Indeed, since $x$ and $x'$ lie on the same stable leaf, and also $y$ and $y'$ lie on the same stable leaf, any admissible curve $\gamma$ connecting $x$ to $y$ must cross at least the family of stable leaves that $u_{x',y'}$ intersects. A similar argument to that in the proof of Lemma \ref{lem:drho} ensures that $d_{\rho}(x,y)\geq\mu^s(u_{x',y'})$. It follows from this inequality, by taking the supremum of $d_{\rho}(x,y)$ with $x,y\in f^n(C)$, that if $f^n(C)$ is contained in a  $V_{i(n)}$ without spines, then $$\diam_{d_\rho}(f^n(C))
\geq \mu^s(\pi^s_n(f^n(C))).$$
%Also, since $f$ preserves the stable foliation $\mc{F}^s$, it follows that
%$$\pi^s_n(f^n(C))=f^n(\pi^s(C)) \,\,\,\,\,\, \text{for every} \,\,\,\,\,\, n\in\N.$$ 
The transverse measure $\mu^s$ is invariant under holonomy along the stable leaves. Therefore, the $\mu^s$-measure of the projection $\pi^s_n(f^n(C))$ depends exclusively on the bundle of stable leaves intersected by $f^n(C)$, making it completely independent of the arbitrary basepoint $x_{i(n)}$ defining the target unstable leaf in $V_{i(n)}$. Because $f$ maps stable leaves to stable leaves and geometrically expands the transverse measure $\mu^s$ by the factor $\lambda$, the measure of this intersected stable leaf bundle scales exactly by $\lambda^n$. This yields the equality:
$$\mu^s(\pi^s_n(f^n(C))) = \lambda^n \mu^s(\pi^s_0(C)).$$
Combining this with our metric estimate inside the charts, we obtain:
$$\diam_{d_\rho}(f^n(C)) \geq \mu^s(\pi^s_n(f^n(C))) = \lambda^n \mu^s(\pi^s(C)).$$
Since $\lambda>1$,  
if we can find such
$n\in\N$ large enough, then
$$\diam_{d_\rho}(f^n(C))\geq\lambda^n\mu^s(\pi^s(C))>\alpha',$$
contradicting (\ref{diamcw}). 
Otherwise, after replacing $C$ by a forward iterate, all subsequent iterates lie in singular charts. The choice of the charts gives $x_{i(n+1)}=f(x_{i(n)})$; passing to the period $k$ of this spine and relabeling $f^k,\lambda^k$ as $f,\lambda$, we may assume $x_{i(n)}=x_{i(0)}$ for all $n\geq0$. The transverse projection remains nontrivial because $f$ preserves the stable foliation.
Note also that $D_0$ defined below satisfies $\mu^s(D_0)=\mu^s(\pi^s(C))>0$,
since $\pi^s(C)$ is a non-trivial sub-continuum of a leaf.
For each $n\in\N\cup\{0\}$, let $$D_n=\pi^s(f^n(C))\subset \mc{F}^u(x_{i(0)})$$
and let $D$ be the closure of the connected component of $x_{i(0)}$ in $\mc{F}^u(x_{i(0)})\cap  V_{i(0)}$. By definition we have $D_n\subset D$ for every $n\in\N$.
Repeating the calculation as before, we see that $\mu^s(D_n)=\lambda^n\mu^s(D_0)\to \infty$ when $n\to+\infty$, while $\mu^s(D)<\infty$. A contradiction, ruling out the case
$f^n(C)\subset  V_{i(0)}$ for every $n\in\N$. The case where $\pi^u(C)$ is a non-trivial continuum is similar, iterating backwards instead. Hence, we proved cw-expansiveness.

The proof that $f$ is cw$_2$-expansive then follows from noting that local stable/unstable continua are contained in small arcs inside leaves of $\mc{F}^s/\mc{F}^u$ and that in small neighborhoods of the points $x_i$, there are at most two distinct points of intersection between any pair of stable/unstable arcs since these neighborhoods are homeomorphic to $\mc{D}_2$ and $\mc{D}_1$. This proves that $f$ is cw$_2$-hyperbolic.
\end{proof}

\section{Classification}

In what follows, we discuss how to classify the surfaces that admit cw$_F$-hyperbolic homeomorphisms and obtain the conjugacy with the linear models on the torus and the affine models on the sphere. First, we explain these affine models on the sphere and note important differences among them.

Let $A \in SL(2,\mathbb{Z})$ be a hyperbolic matrix and consider the classical involution $\sigma \colon \mathbb{T}^2 \to \mathbb{T}^2$ given by $\sigma(x) = -x$. The quotient space $\mathbb{S}^2 \cong \mathbb{T}^2 / \sigma$ has four branch points that correspond to the projections of the four 2-torsion points of $\mathbb{T}^2$, namely, the set $$\mathcal{P} = \{ (0,0), (1/2,0), (0,1/2), (1/2,1/2) \}.$$ 
An affine diffeomorphism $f \colon \mathbb{T}^2 \to \mathbb{T}^2$ of the form $f(x) = Ax + t \pmod{\mathbb{Z}^2}$ commutes with $\sigma$ if, and only if, $2t \equiv 0 \pmod{\mathbb{Z}^2}$, that is, if $t \in \mathcal{P}$. In this case, $f$ induces a cw$_F$-hyperbolic homeomorphism on the sphere whose dynamics on the singular spines depend on the choice of $A$ and $t$.

\begin{enumerate}
    \item \textit{Four fixed spines.} 
    Let $A = \begin{pmatrix} 5 & 2 \\ 2 & 1 \end{pmatrix}$ and take the zero translation $t = (0,0)$. Since $A \equiv I \pmod 2$, the matrix strictly fixes the four 2-torsion points modulo $\mathbb{Z}^2$. Consequently, the linear model induced on the sphere has four fixed spines.
    
    \item \textit{Two fixed spines and one cycle of period 2.} 
    Let $A = \begin{pmatrix} 2 & 3 \\ 1 & 2 \end{pmatrix}$ and $t = (0,0)$. The origin $(0,0)$ and the vector $(1/2, 1/2)$ are preserved by $A$ modulo $\mathbb{Z}^2$, generating two fixed spines on $\mathbb{S}^2$. However, the matrix acts on the remaining points by mutually permuting them: 
    $$\hspace{+0.8cm}A(1/2, 0) \equiv (0, 1/2) \pmod{\mathbb{Z}^2} \quad \text{and} \quad A(0, 1/2) \equiv (1/2, 0) \pmod{\mathbb{Z}^2}.$$ This induces a cycle of period 2 between the two remaining spines.
    
    \item \textit{One fixed spine and one cycle of period 3.} 
    Let $A = \begin{pmatrix} 2 & 1 \\ 1 & 1 \end{pmatrix}$ and $t = (0,0)$. The automorphism preserves the origin, inducing one fixed spine on the sphere. The other three points form a cyclic orbit: 
    $$(1/2, 0) \mapsto (0, 1/2) \mapsto (1/2, 1/2) \mapsto (1/2, 0) \pmod{\mathbb{Z}^2},$$ resulting in a cycle of spines of period 3.
    
    \item \textit{Two cycles of period 2.} 
    Let $A = \begin{pmatrix} 5 & 2 \\ 2 & 1 \end{pmatrix}$ and take the non-zero translation $t = (1/2, 0)$. Evaluating the action of $f(x) = Ax + t$ on the branch points, we obtain:
    $$f(0,0) = (1/2, 0) \quad \text{and} \quad f(1/2, 0) = (3, 1) \equiv (0,0) \pmod{\mathbb{Z}^2},$$
    $$f(0, 1/2) = (3/2, 1/2) \equiv (1/2, 1/2) \pmod{\mathbb{Z}^2},$$ $$\text{and} \quad f(1/2, 1/2)=(4,3/2) \equiv (0, 1/2) \pmod{\mathbb{Z}^2}.$$
    The affine map divides the four branch points into two disjoint cycles of period 2, so that the cwF-hyperbolic map has no fixed spine.
    
    \item \textit{One cycle of period 4.} 
    Let $A = \begin{pmatrix} 2 & 3 \\ 1 & 2 \end{pmatrix}$ with the translation $t = (1/2, 0)$. The affine action $f(x) = Ax + t$ on $\mathcal{P}$ results in a single closed cycle:
    $$ (0,0) \mapsto (1/2, 0) \mapsto (1/2, 1/2) \mapsto (0, 1/2) \mapsto (0,0) \pmod{\mathbb{Z}^2}.$$
    Therefore, the four spines on $\mathbb{S}^2$ form a single orbit of period 4.
\end{enumerate}

Examples as above belonging to distinct cases cannot be topologically conjugate since a conjugacy would preserve the spines and their periods. Moreover, the affine examples (4) and (5) are not linear and have no fixed spines, which ensures that they cannot be conjugate to linear examples that necessarily have fixed spines.

Now we proceed to the proof of Theorem \ref{A}.

\begin{proof}[Proof of Theorem \ref{A}]
First, since the boundary of a surface must be an invariant set of a homeomorphism and there are no cw-expansive homeomorphisms on a one-dimensional manifold (this is a consequence of \cite{K2}*{Theorem 1.6}), it is clear that $S$ is a surface without boundary. By a well-known classification result (e.g. see Theorem 7.2 in Chapter 1 of \cite{Massey}) any closed and orientable surface is either homeomorphic to a sphere or to a connected sum of tori, while any closed and non-orientable surface is homeomorphic to the connected sum of either a projective plane or Klein bottle and a compact, orientable surface.

Let $f\colon S\to S$ be a cw$_F$-hyperbolic surface homeomorphism. Its set of spines is finite by Theorem \ref{finitespines}. If $\spine(f)=\emptyset$, then Theorem~\ref{foliation} ensures that $f$ is topologically hyperbolic, so $f$ is topologically conjugate to an Anosov automorphism of $\mathbb{T}^2$.

For the remaining case, assume that $\spine(f)\neq\emptyset$.
By Theorem \ref{foliation} there exists an Anosov automorphism $\tilde f$ of $\mathbb{T}^2$ induced by a hyperbolic matrix $A\in GL(2,\Z)$ and a factor map $\pi\colon (\mathbb{T}^2,\tilde f) \to (S,f)$ such that $\pi$ is a branched double cover and $\# \pi^{-1}(p)=1$ if, and only if, $p\in \spine(f)$.
Note that the Riemann-Hurwitz formula \cite{J} gives us 
$$0=\chi(\mathbb{T}^2)=2\chi(S)-\#\spine(f).$$ 
Thus, the above equality reduces to
\begin{equation}
    \chi(S)=\frac{1}{2}\# \spine(f).
    \label{euler:spines}
\end{equation}
This ensures, in particular, that $\chi(S)\geq0$ and that there are no cw$_F$-hyperbolic homeomorphisms on surfaces with negative characteristic. Combining this with equality \eqref{euler:spines} (cf. Theorem 8.2 in Chapter 1 of \cite{Massey}) we obtain that $S$ is either a sphere with $\chi(S)=2$, or a projective plane with $\chi(S)=1$, or the torus with $\chi(S)=0$, or the Klein bottle with $\chi(S)=0$.

If $\chi(S)=0$, then equality (\ref{euler:spines}) ensures that $\spine(f)=\emptyset$ and we can repeat the argument from the proof of Theorem~\ref{foliation} obtaining that $f$ is topologically hyperbolic. The classification of topologically hyperbolic surface homeomorphisms allows us to conclude that on the torus $f$ is conjugate to an Anosov automorphism while on the Klein bottle there are no cw$_F$-hyperbolic homeomorphisms. 

Now assume that $S=\mathbb{S}^2$ is a sphere with $\chi(S)=2$. In this case, it follows from (\ref{euler:spines}) that there exist exactly four distinct spines of $f$. Recall that by Theorem \ref{foliation}, $f$ is a factor of an Anosov automorphism $(\T^2,\tilde f)$ by a factor map $\pi$.  Hereafter, $\sigma\colon\T^2\to\mathbb S^2$ denotes the quotient projection.
%Also, $\tilde f$ induces a homeomorphism $g\colon \mathbb{S}^2\to\mathbb{S}^2$, using the antipodal factor map, denoted by $\sigma\colon\mathbb{T}^2\to\mathbb{S}^2$. 
Thus, both $\pi$ and $\sigma$ are branched double covering maps from $\mathbb{T}^2$ to $\mathbb{S}^2$ with exactly four branch points, $\spine(f)$ and $\sigma(\mathcal{P})$. 
If we remove the spines and their pre-images, that is, we let $$\mathbb{S}^2_{\pi}=\mathbb{S}^2\setminus\spine(f), \,\,\,\,\,\, \mathbb{T}^{2}_{\pi}=\mathbb{T}^2\setminus\pi^{-1}(\spine(f))$$ 
$$\mathbb{S}^2_{\sigma}=\mathbb{S}^2\setminus\sigma(\mathcal{P})\,\,\,\,\,\, \text{and} \,\,\,\,\,\, \mathbb{T}^{2}_{\sigma}=\mathbb{T}^2\setminus\mathcal{P},$$ we have that both $\pi\colon\mathbb{T}^{2}_{\pi}\to\mathbb{S}^2_{\pi}$ and $\sigma\colon\mathbb{T}^{2}_{\sigma}\to\mathbb{S}^2_{\sigma}$ are double covering maps. 

We will prove that there are homeomorphisms $h\colon\mathbb{T}^2\to \mathbb{T}^2$ and $\psi\colon\mathbb{S}^2\to \mathbb{S}^2$ such that $\psi\circ\pi=\sigma\circ h$. Thus, $h$ sends fibers of $\pi$ bijectively to fibers of $\sigma$, with the
corresponding base map $\psi$. It suffices to conjugate the deck transformations, that is, to prove $h\circ \tau_{\pi}=\tau_{\sigma}\circ h$. 
%$\pi$ and $\sigma$ are isomorphic, that is, there exists a homeomorphism $h\colon\mathbb{T}^2\to \mathbb{T}^2$ such that $\sigma\circ h=\pi$. This means that $h$ sends fibers of $\pi$ into fibers of $\sigma$, i.e., if $x\in\pi^{-1}(y)$, then $h(x)\in\sigma^{-1}(y)$ since $\sigma(h(x))=\pi(x)=y$. This is equivalent to proving that $h$ conjugates the deck transformations of $\pi$ and $\sigma$, that is, $h\circ\tau_{\pi}=\tau_{\sigma}\circ h$.
Deck transformations $\tau_{\pi}$ of $\pi$ and $\tau_{\sigma}$ of $\sigma$ can be defined as follows: for each $x\in\mathbb{T}^2_{\pi}$, we let $\tau_{\pi}(x)$ be the point in $\pi^{-1}(\pi(x))$ that is distinct from $x$, and for each $x\in\mathbb{T}^2_{\sigma}$, we let $\tau_{\sigma}(x)$ be the point in $\sigma^{-1}(\sigma(x))$ that is distinct from $x$ (recall that $\pi\colon\mathbb{T}^{2}_{\pi}\to\mathbb{S}^2_{\pi}$ and $\sigma\colon\mathbb{T}^{2}_{\sigma}\to\mathbb{S}^2_{\sigma}$ are double covering maps, so $\pi^{-1}(\pi(x))$ and $\sigma^{-1}(\sigma(x))$ contain exactly two points). Extend continuously both maps to $\mathbb{T}^2$ by letting $\tau_{\pi}(x)=x$ for each $x\in\pi^{-1}(\spine(f))$ and $\tau_{\sigma}(x)=x$ for each $x\in\mathcal{P}$. Thus, the maps $\tau_{\pi}$ and $\tau_{\sigma}$ have exactly four fixed points, which are the branching points of the covering maps $\pi$ and $\sigma$ on $\T^2$. Also, $\tau_{\pi}$ and $\tau_{\sigma}$ 
are involutions, that is, $\tau_{\pi}^2=id$ and $\tau_{\sigma}^2=id$. More precisely, since $\sigma$ identifies antipodal points, it follows that $\tau_{\sigma}$ lifts to $\R^2$ as $\tilde\tau_{\sigma}=-\text{id}$. 

We prove that $\tau_{\pi}$ lifts to $\R^2$ as $\tilde\tau_{\pi}(x)=-x+q$ with $q\in\R^2$. Since $\tau_{\pi}$ is a homeomorphism of $\mathbb{T}^2$, and its lift $\tilde \tau_{\pi}\colon\R^2\to\R^2$ maps fibers of the universal covering onto fibers, we have 
$$\tilde\tau_{\pi}(x+k)-\tilde\tau_{\pi}(x)\in \Z^2 \,\,\,\,\,\, \text{for every} \,\,\,\,\,\, x\in\R^2 \,\,\,\,\,\, \text{and}  \,\,\,\,\,\, k\in \Z^2.$$ This implies that there is a matrix $Q\in GL(2,\Z)$ such that
$$\tilde\tau_{\pi}(x+k)-\tilde\tau_{\pi}(x)=Qk\,\,\,\,\,\, \text{for every} \,\,\,\,\,\, x\in\R^2 \,\,\,\,\,\, \text{and}  \,\,\,\,\,\, k\in \Z^2.$$
It follows that
\begin{eqnarray*}
\tilde\tau^2_{\pi}(x+k)-\tilde\tau^2_{\pi}(x)&=&
\tilde\tau_{\pi}(\tilde\tau_{\pi}(x+k))-\tilde\tau_{\pi}(\tilde\tau_{\pi}(x)+k)+\tilde\tau_{\pi}(\tilde\tau_{\pi}(x)+k)-\tilde\tau_{\pi}(\tilde\tau_{\pi}(x))\\
&=&\tilde\tau_{\pi}(\tilde\tau_{\pi}(x)+k+Qk-k)-\tilde\tau_{\pi}(\tilde\tau_{\pi}(x)+k)+Qk\\
&=&Q(Qk-k)+Qk\\
&=&Q^2 k.
\end{eqnarray*}
But $\tilde\tau_{\pi}^2$ is a lift of the identity, so $\tilde\tau_{\pi}^2(x)-x=c\in \Z^2$ for every $x\in \R^2$. In particular,
\begin{eqnarray*}
\tilde\tau^2_{\pi}(x+k)-\tilde\tau^2_{\pi}(x)&=&
\tilde\tau^2_{\pi}(x+k)-(x+k)-\tilde\tau^2_{\pi}(x)+x+k\\
&=&c-c+k\\
&=&k.
\end{eqnarray*}
Then we have $Q^2k=k$ for every $k\in \Z^2$ and therefore $Q^2=I$ (by linearity of $Q^2$).
Let 
$P(x)=\tilde\tau_{\pi}(x)-Qx$ and note that
$$P(x+k)=\tilde\tau_{\pi}(x)+Qk-Q(x+k)=P(x) \,\,\,\,\,\, \text{for every} \,\,\,\,\,\, x\in\R^2 \,\,\,\,\,\, \text{and} \,\,\,\,\,\, k\in \Z^2,$$
which proves, in particular, that $P$ is bounded. 
Note that $\tau_{\pi}$ is commuting with the toral automorphism $\tilde f$ defined by the hyperbolic matrix $A$, and this ensures that
$$A(\tilde\tau_{\pi}(x))-\tilde\tau_{\pi}(A(x))=b \,\,\,\,\,\, \text{for some} \,\,\,\,\,\, b\in \Z^2.$$
Since $\tilde\tau_{\pi}(x)=Qx+P(x)$, it follows that
$$(AQ-QA)x=P(A(x))-AP(x)+b \,\,\,\,\,\, \text{for every} \,\,\,\,\,\, x\in \R^2.$$
But since $P$ is bounded, the right-hand side of the above equation is also bounded, 
which is possible only if $AQ=QA$. This ensures that $b=AP(x)-P(A(x))$ for every $x\in\R^2$, hence $b=AP(0)-P(A(0))=AP(0)-P(0)$. Let $R(x)=P(x)-P(0)$ and note that
\begin{eqnarray*}
AR(x)&=&AP(x)-AP(0)\\
&=&P(A(x))+b-AP(0)\\
&=&P(A(x))-P(0)\\
&=&R(A(x)).
\end{eqnarray*}
This ensures that 
\begin{equation}\label{AnR}
A^nR(x)=R(A^n(x)) \,\,\,\,\,\, \text{for every} \,\,\,\,\,\, n\in \Z.
\end{equation}
But since $A$ is a hyperbolic matrix and $R$ is bounded, we must have $R\equiv 0$
and consequently $P\equiv P(0)$ and $\tilde\tau_{\pi}(x)=Qx+P(0)$.
Indeed, if this were not the case, then denoting $v=R(x)\neq 0$ we either have
$$\lim_{n\to \infty}||A^nv||=+\infty \,\,\,\,\,\, \text{or} \,\,\,\,\,\, \lim_{n\to\infty}||A^{-n}v||=+\infty,$$ which contradicts (\ref{AnR}) and the fact that $R$ is bounded.
Thus, $\tau_{\pi}$ is an affine torus homeomorphism 
with four fixed points, so $Q-I$ is nonsingular and $|\det(Q-I)|=4$. Since additionally $Q^2=I$, we must have $Q=-I$. This proves that 
$$\tilde\tau_{\pi}(x)=-x+P(0) \,\,\,\,\,\, \text{for every} \,\,\,\,\,\, x\in\R^2,$$
that is, $\tilde\tau_{\pi}(x)=-x+q$ with $q=P(0)\in\R^2$.

Now define $\tilde h\colon \mathbb{R}^{2}\to \mathbb{R}^{2}$ by $\tilde h(x)=x-\frac{q}{2}$. Note that $\tilde h\circ\tilde\tau_{\pi}=\tilde\tau_{\sigma}\circ\tilde h$ because
$$\tilde h(\tilde{\tau}_{\pi}(x))=\tilde h(-x+q)=(-x+q)-\frac{q}{2}=-x+\frac{q}{2} \,\,\,\,\,\, \text{and}$$
$$\tilde{\tau}_{\sigma}(\tilde h(x))=\tilde{\tau}_{\sigma}\left(x-\frac{q}{2}\right)=-\left(x-\frac{q}{2}\right)=-x+\frac{q}{2}.$$
It follows that $\tilde h$ induces a homeomorphism $h\colon \mathbb{T}^{2}\to \mathbb{T}^{2}$ satisfying $h\circ\tau_{\pi}=\tau_{\sigma}\circ h$ and, consequently, sending fibers of $\pi$ bijectively to fibers of $\sigma$. 
This means that $h$ induces
an isomorphism of the coverings, with a base homeomorphism to be constructed below. In particular, $h$ sends the four branching points of $\pi$ to the four branching points of $\sigma$. Since $\tilde h$ is a translation by $-\frac{q}{2}$, $\tilde h(\frac{q}{2})=(0,0)$, and the branching points of $\sigma$ are the projections of the points $\{(0,0), (1/2, 0), (0, 1/2), (1/2, 1/2)\}$, it follows that the branching points of $\pi$ are projections of the points $\{q/2, q/2+(1/2, 0), q/2+(0, 1/2), q/2+(1/2, 1/2)\}$.

Let $F\colon\mathbb{T}^2\to\mathbb{T}^2$ be defined as $F:=h\circ\tilde{f}\circ h^{-1}$. We prove that
$$F(x)=A(x)+p, \qquad \text{with} \qquad p=\frac{(A-I)q}{2}.$$
Indeed, by lifting this composition to $\mathbb{R}^2$ and evaluating it at a point $x \in \mathbb{R}^2$, we obtain:
\begin{align*}
    \tilde{F}(x) &= \tilde{h}\left(A\left(\tilde{h}^{-1}(x)\right)\right) \\
    &= \tilde{h}\left(A\left(x + \frac{q}{2}\right)\right) \\
    &= A\left(x + \frac{q}{2}\right) - \frac{q}{2} \\
    &= Ax + \frac{(A-I)q}{2}.
\end{align*}
Thus, on the torus, we have $F(x) = Ax + p \pmod{\mathbb{Z}^2}$.

To prove that $2p = 0$ on $\mathbb{T}^2$, we use the relation $A \circ \tilde{\tau}_\pi = \tilde{\tau}_\pi \circ A + b$ for some $b \in \mathbb{Z}^2$. Substituting $\tilde{\tau}_\pi(x) = -x + q$ into both sides of this identity yields:
\begin{align*}
    A(\tilde{\tau}_\pi(x)) &= A(-x + q) = -Ax + Aq \\
    \tilde{\tau}_\pi(A(x)) + b &= -(Ax) + q + b = -Ax + q + b.
\end{align*}
Equating these expressions gives $-Ax + Aq = -Ax + q + b$, which simplifies directly to $(A - I)q = b$. Substituting this result into our expression for $p$, we find:
$$ p = \frac{(A-I)q}{2} = \frac{b}{2}. $$
Since $b \in \mathbb{Z}^2$, it immediately follows that $2p = b \equiv 0 \pmod{\mathbb{Z}^2}$. Consequently, $F$ commutes with $\tau_{\sigma}$ and induces a homeomorphism $g\colon\mathbb{S}^2\to\mathbb{S}^2$ satisfying $\sigma\circ F=g\circ\sigma$.

This is enough to obtain a homeomorphism 
$\psi\colon\mathbb{S}^2\to\mathbb{S}^2$ such that $\psi\circ\pi = \sigma\circ h$.
Indeed, for each $y\in\mathbb{S}^2$ and $x\in \pi^{-1}(y)$, let $\psi(y)=\sigma(h(x))$. If $y\in\spine(f)$, then $\pi^{-1}(y)$ is a single point, hence $\psi(y)$ is well-defined. When $y\in\mathbb{S}^{2}_{\pi}$, $\psi(y)$ does not depend on the choice of $x\in \pi^{-1}(y)$ because $\pi^{-1}(y)=\{x,\tau_{\pi}(x)\}$ and $\sigma(h(x))=\sigma(h(\tau_{\pi}(x)))$ since $h$ is an isomorphism between the coverings $\pi$ and $\sigma$.
Now we prove that $\psi$ is a homeomorphism. Since $\sigma\colon\mathbb{T}^2 \to \mathbb{S}^2$ and $h\colon\mathbb{T}^2 \to \mathbb{T}^2$ are surjective, for each $z \in \mathbb{S}^2$, there exists $x \in \mathbb{T}^2$ such that $\sigma(h(x)) = z$, and since $\psi \circ \pi = \sigma\circ h$, it follows that $\psi(\pi(x)) = z$, so $\psi$ is surjective. Suppose $\psi(y_1) = \psi(y_2)$ and consider $x_1 \in \pi^{-1}(y_1)$ and $x_2 \in \pi^{-1}(y_2)$. Then $$\sigma(h(x_1))=\psi(\pi(x_1))=\psi(y_1)=\psi(y_2)=\psi(\pi(x_2))=\sigma(h(x_2)).$$ 
This means that $h(x_1)$ and $h(x_2)$ are in the same fiber of $\sigma$, which ensures that $x_1$ and $x_2$ are in the same fiber of $\pi$. Therefore, $\pi(x_1)=\pi(x_2)$, which means $y_1=y_2$. This proves injectivity of $\psi$. The continuity of $\psi$ follows from the continuity of $\pi$, $\sigma$ and $h$, and the equality $\psi\circ\pi = \sigma\circ h$.

We conclude this case by proving that $\psi$ conjugates $f$ and $g$. Indeed,
by hypothesis, we have the factor relations:
\begin{equation}
\pi \circ \tilde f = f \circ \pi \quad \text{and} \quad \sigma \circ F = g \circ \sigma.
\end{equation}
We wish to show that $\psi \circ f = g \circ \psi$. Let $y \in \mathbb{S}^2$ and $x \in\mathbb{T}^2$ be such that $\pi(x) = y$. Then we have
\begin{align*}
(\psi \circ f)(y) &= \psi(f(\pi(x))) \\
&= \psi(\pi(\tilde f(x))) \tag{by the factor property of $f$} \\
&= \sigma(h(\tilde f(x))) \tag{since $\psi \circ \pi = \sigma\circ h$} \\
&= \sigma(F(h(x))) \tag{since $h\circ\tilde f=F\circ h$} \\
&= g(\sigma(h(x))) \tag{by the factor property of $g$} \\
&= g(\psi(\pi(x))) \tag{since $\sigma\circ h = \psi \circ \pi$} \\
&= (g \circ \psi)(y).
\end{align*}
Since this holds for all $y \in\mathbb{S}^2$, we have $\psi \circ f = g \circ \psi$. This proves that all cw$_F$-hyperbolic homeomorphisms of the sphere are topologically conjugate to the standard hyperelliptic quotient of an affine map of the form $x\to Ax+p$, where $A\in GL(2,\Z)$ is hyperbolic and $2p=0\hspace{-0.1cm}\mod\Z^2$. The diagram in Figure \ref{fig3} illustrates the above equalities and commutations.

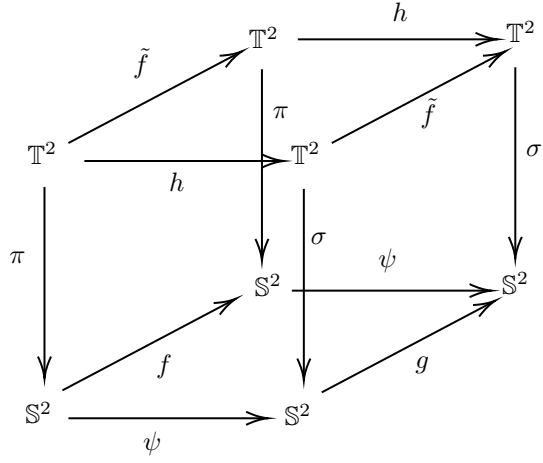
\begin{figure}
	\centering
    \tikzset{every picture/.style={line width=0.75pt}} %set default line width to 0.75pt
    \begin{tikzpicture}[x=0.75pt,y=0.75pt,yscale=-1,xscale=1]
    \draw    (192,198.25) -- (277.23,153.18) ;
    \draw [shift={(279,152.25)}, rotate = 152.13] [color={rgb, 255:red, 0; green, 0; blue, 0 }  ][line width=0.75]    (10.93,-3.29) .. controls (6.95,-1.4) and (3.31,-0.3) .. (0,0) .. controls (3.31,0.3) and (6.95,1.4) .. (10.93,3.29)   ;
    \draw    (322,199.25) -- (407.23,154.18) ;
    \draw [shift={(409,153.25)}, rotate = 152.13] [color={rgb, 255:red, 0; green, 0; blue, 0 }  ][line width=0.75]    (10.93,-3.29) .. controls (6.95,-1.4) and (3.31,-0.3) .. (0,0) .. controls (3.31,0.3) and (6.95,1.4) .. (10.93,3.29)   ;
    \draw    (195,212.25) -- (293,212.25) ;
    \draw [shift={(295,212.25)}, rotate = 180] [color={rgb, 255:red, 0; green, 0; blue, 0 }  ][line width=0.75]    (10.93,-3.29) .. controls (6.95,-1.4) and (3.31,-0.3) .. (0,0) .. controls (3.31,0.3) and (6.95,1.4) .. (10.93,3.29)   ;
    \draw    (307,148.25) -- (405,148.25) ;
    \draw [shift={(407,148.25)}, rotate = 180] [color={rgb, 255:red, 0; green, 0; blue, 0 }  ][line width=0.75]    (10.93,-3.29) .. controls (6.95,-1.4) and (3.31,-0.3) .. (0,0) .. controls (3.31,0.3) and (6.95,1.4) .. (10.93,3.29)   ;
    \draw    (183,96.25) -- (183,190.25) ;
    \draw [shift={(183,192.25)}, rotate = 270] [color={rgb, 255:red, 0; green, 0; blue, 0 }  ][line width=0.75]    (10.93,-3.29) .. controls (6.95,-1.4) and (3.31,-0.3) .. (0,0) .. controls (3.31,0.3) and (6.95,1.4) .. (10.93,3.29)   ;
    \draw    (195,74.25) -- (280.23,29.18) ;
    \draw [shift={(282,28.25)}, rotate = 152.13] [color={rgb, 255:red, 0; green, 0; blue, 0 }  ][line width=0.75]    (10.93,-3.29) .. controls (6.95,-1.4) and (3.31,-0.3) .. (0,0) .. controls (3.31,0.3) and (6.95,1.4) .. (10.93,3.29)   ;
    \draw    (292,37.25) -- (292,131.25) ;
    \draw [shift={(292,133.25)}, rotate = 270] [color={rgb, 255:red, 0; green, 0; blue, 0 }  ][line width=0.75]    (10.93,-3.29) .. controls (6.95,-1.4) and (3.31,-0.3) .. (0,0) .. controls (3.31,0.3) and (6.95,1.4) .. (10.93,3.29)   ;
    \draw    (203,83.25) -- (301,83.25) ;
    \draw [shift={(303,83.25)}, rotate = 180] [color={rgb, 255:red, 0; green, 0; blue, 0 }  ][line width=0.75]    (10.93,-3.29) .. controls (6.95,-1.4) and (3.31,-0.3) .. (0,0) .. controls (3.31,0.3) and (6.95,1.4) .. (10.93,3.29)   ;
    \draw    (310,22.25) -- (408,22.25) ;
    \draw [shift={(410,22.25)}, rotate = 180] [color={rgb, 255:red, 0; green, 0; blue, 0 }  ][line width=0.75]    (10.93,-3.29) .. controls (6.95,-1.4) and (3.31,-0.3) .. (0,0) .. controls (3.31,0.3) and (6.95,1.4) .. (10.93,3.29)   ;
    \draw    (313,97.25) -- (313,191.25) ;
    \draw [shift={(313,193.25)}, rotate = 270] [color={rgb, 255:red, 0; green, 0; blue, 0 }  ][line width=0.75]    (10.93,-3.29) .. controls (6.95,-1.4) and (3.31,-0.3) .. (0,0) .. controls (3.31,0.3) and (6.95,1.4) .. (10.93,3.29)   ;
    \draw    (419,36.25) -- (419,130.25) ;
    \draw [shift={(419,132.25)}, rotate = 270] [color={rgb, 255:red, 0; green, 0; blue, 0 }  ][line width=0.75]    (10.93,-3.29) .. controls (6.95,-1.4) and (3.31,-0.3) .. (0,0) .. controls (3.31,0.3) and (6.95,1.4) .. (10.93,3.29)   ;
    \draw    (327,74.25) -- (412.23,29.18) ;
    \draw [shift={(414,28.25)}, rotate = 152.13] [color={rgb, 255:red, 0; green, 0; blue, 0 }  ][line width=0.75]    (10.93,-3.29) .. controls (6.95,-1.4) and (3.31,-0.3) .. (0,0) .. controls (3.31,0.3) and (6.95,1.4) .. (10.93,3.29)   ;
    \draw (237.5,178.65) node [anchor=north west][inner sep=0.75pt]    {$f$};
    \draw (367.5,179.65) node [anchor=north west][inner sep=0.75pt]    {$g$};
    \draw (231,217.4) node [anchor=north west][inner sep=0.75pt]    {$\psi $};
    \draw (349,125.4) node [anchor=north west][inner sep=0.75pt]    {$\psi $};
    \draw (284,14.4) node [anchor=north west][inner sep=0.75pt]    {$\mathbb{T}^{2}$};
    \draw (173,72.4) node [anchor=north west][inner sep=0.75pt]    {$\mathbb{T}^{2}$};
    \draw (296,54.4) node [anchor=north west][inner sep=0.75pt]    {$\pi $};
    \draw (164,127.4) node [anchor=north west][inner sep=0.75pt]    {$\pi $};
    \draw (315,117.4) node [anchor=north west][inner sep=0.75pt]    {$\sigma $};
    \draw (423,73.4) node [anchor=north west][inner sep=0.75pt]    {$\sigma $};
    \draw (370,50.4) node [anchor=north west][inner sep=0.75pt]    {$F$};
    \draw (244,87.4) node [anchor=north west][inner sep=0.75pt]    {$h$};
    \draw (356,2.4) node [anchor=north west][inner sep=0.75pt]    {$h$};
    \draw (172,202.4) node [anchor=north west][inner sep=0.75pt]    {$\mathbb{S}^{2}$};
    \draw (302,203.4) node [anchor=north west][inner sep=0.75pt]    {$\mathbb{S}^{2}$};
    \draw (287,138.4) node [anchor=north west][inner sep=0.75pt]    {$\mathbb{S}^{2}$};
    \draw (411,137.4) node [anchor=north west][inner sep=0.75pt]    {$\mathbb{S}^{2}$};
    \draw (305,72.4) node [anchor=north west][inner sep=0.75pt]    {$\mathbb{T}^{2}$};
    \draw (413,11.4) node [anchor=north west][inner sep=0.75pt]    {$\mathbb{T}^{2}$};
    \draw (226,25.4) node [anchor=north west][inner sep=0.75pt]    {$\tilde{f}$};
    \end{tikzpicture}
    \caption{Commutative diagram.}
    \label{fig3}
\end{figure}

Now assume that $S$ is the projective plane with $\chi(S)=1$. 
We claim that there does not exist a cw$_2$-hyperbolic homeomorphism $\bar f\colon \mathbb{RP}^2\to \mathbb{RP}^2$. 
Suppose, on the contrary, the existence of a cw$_2$-hyperbolic homeomorphism $\bar f\colon \mathbb{RP}^2\to \mathbb{RP}^2$. We can lift $\bar f$ to a cw$_2$-hyperbolic homeomorphism $f\colon \mathbb{S}^2\to \mathbb{S}^2$ using a double covering map $\alpha\colon\mathbb{S}^2\to\mathbb{RP}^2$ defined by the relation $x\sim -x$ on the unit sphere $\mathbb{S}^2\subset \R^3$ (see Theorem \ref{fincovcw} in the appendix and Corollary~\ref{cor:cw2}). Then by the case of cw$_F$-hyperbolic homeomorphisms of $\mathbb{S}^2$ considered earlier, as we have already proven, there is a hyperbolic affine map on $\mathbb{T}^2$ of the form $x\to Ax+p$ which quotients to a homeomorphism $g\colon \mathbb{S}^2\to \mathbb{S}^2$ using the antipodal factor map $\sigma\colon \mathbb{T}^2\to \mathbb{S}^2$,
and a homeomorphism $\psi\colon \mathbb{S}^2\to \mathbb{S}^2$ satisfying $g\circ \psi=\psi\circ f$ and $f\circ \psi^{-1}=\psi^{-1}\circ g$. The deck transformation of $\alpha$ is the antipodal map $\tau_{\alpha}(x)=-x$ of $\R^3$ restricted to $\mathbb{S}^2$. Note that it is a fixed-point-free involution on $\mathbb{S}^2$.
Consider the map $i\colon\mathbb{S}^2\to\mathbb{S}^2$ defined by $i:=\psi\circ \tau_{\alpha}\circ \psi^{-1}$ and
note that $i^2=\psi\circ\tau_{\alpha}^2\circ\psi^{-1}=id$, that is, $i$ is an involution of $\mathbb{S}^2$. Note that $i$ is also fixed point free since $i$ is topologically conjugate to $\tau_{\alpha}$ by $\psi^{-1}$. Moreover, note that $i$ commutes with $g$ since
\begin{eqnarray*}
i\circ g&=&\psi\circ \tau_{\alpha}\circ \psi^{-1}\circ g\\
&=& \psi\circ \tau_{\alpha}\circ f\circ \psi^{-1} \\
&=&\psi\circ f\circ \tau_{\alpha}\circ \psi^{-1}\\
&=&  g\circ \psi\circ\tau_{\alpha}\circ \psi^{-1}\\
&=&g\circ i.
\end{eqnarray*}

We claim that such an involution $i$ cannot exist. The set $\spine(g)$ is $g$-invariant, and $i\circ g=g \circ i$, which forces $i(\spine(g))$ to be a
$g$-invariant set with the same local $1$-prong structure, that is, the set $\spine(g)$ itself. As
$\spine(g)$ is exactly the set of branching points of $\sigma\colon\mathbb{T}^2\to\mathbb{S}^2$, the involution $i$
preserves the branch data of $\sigma$ and its monodromy, which sends a simple loop around each branch point to $1\in\Z/2\Z$. It therefore lifts to a homeomorphism $\widehat i$ of
$\T^2$ with $\sigma\circ\widehat i=i\circ\sigma$, which lifts further to
$\widetilde i\colon\R^2\to\R^2$ given by $\widetilde i(x)=Lx+P(x)$, where $L\in GL(2,\Z)$ and $P$ is continuous and $\Z^2$-periodic, and we denote the affine map on $\T^2$ by $F$ and its lift by $\widetilde F(x)=Ax+p$. Since
$$\sigma\circ\widehat i\circ F=i\circ g\circ\sigma
=g\circ i\circ\sigma=\sigma\circ F\circ\widehat i,$$
the two compositions differ by a deck transformation away from the branch points, hence everywhere by continuity. Lifting to $\R^2$ gives
$$\widetilde i\circ\widetilde F=s\,\widetilde F\circ\widetilde i+k,
\qquad s\in\{1,-1\},\quad k\in\Z^2.$$
Thus $LA=sAL$, but the case $s=-1$ gives $\operatorname{tr}A=0$, contradicting hyperbolicity since $\det A=\pm1$. Hence $s=1$ and $P(\widetilde F(x))=AP(x)+b$, where $b=(I-L)p+k$. Setting $r=(I-A)^{-1}b$ and $R=P-r$, we obtain for all $n\in\Z$ that
$$R(\widetilde F^n(x))=A^nR(x).$$
Boundedness of $R$ and hyperbolicity of $A$ give $R=0$, so $\widetilde i(x)=Lx+r$.

Furthermore, $\tilde{i}$ must be compatible with the covering involution $\tau_\sigma(x)=-x$, meaning $\tilde{i}(-x) \equiv -\tilde{i}(x) \pmod{\mathbb{Z}^2}$. This implies
$$ -Lx + r \equiv -Lx - r \pmod{\mathbb{Z}^2} \implies 2r \equiv 0 \pmod{\mathbb{Z}^2}.$$
%Hence, $p$ is a point of order 2 in the torus.
In other words, $r\in \frac{1}{2}\Z^2$. Since $i$ is fixed point free on $\mathbb{S}^2$, its lift $\tilde{i}$ cannot map any $x \in \mathbb{R}^2$ to its own class or its antipodal class modulo $\mathbb{Z}^2$. This gives two conditions for all $x \in \mathbb{R}^2$ that have to be satisfied simultaneously:
\begin{enumerate}
    \item $Lx + r %\not\equiv 
    \neq x \pmod{\mathbb{Z}^2} \implies (L - I)x %\not\equiv 
    \neq -r $.%\pmod{\mathbb{Z}^2}$.
    \item $Lx + r %\not\equiv 
    \neq -x \pmod{\mathbb{Z}^2} \implies (L + I)x %\not\equiv 
    \neq -r $.%\pmod{\mathbb{Z}^2}$.
\end{enumerate}
For these equations to have no solutions in $\mathbb{T}^2$, the linear operators $(L - I)$ and $(L + I)$ must be singular. Consequently, $\det(L - I) = 0$ and $\det(L + I) = 0$, which implies that $1$ and $-1$ are eigenvalues of $L$. Since $L \in GL(2, \mathbb{Z})$, $L$ must be a reflection with $\det(L) = -1$ and $\text{tr}(L)=0$.
Also note that there are eigenvectors of $L$ with (rational and therefore also with) integer coordinates.
It is also clear that $L^2=I$, so in particular $L=L^{-1}$. 

Since $AL=LA$, the line $\ker(L-I)$ is $A$-invariant and contains a nonzero rational vector. Thus $A$ has a rational eigenvalue, necessarily $\pm1$ since $A\in GL(2,\Z)$, contradicting hyperbolicity.

We have just proved that the involution $i$ cannot exist, and as a consequence, there is no cw$_2$-hyperbolic homeomorphism
on the real projective plane $\mathbb{RP}^2$. This proves the claim and at the same time completes the proof of the theorem.
\end{proof}

\section{Appendix}

In this appendix, we prove that the lifts of a cw$_F$-hyperbolic homeomorphism by a finite covering map are cw$_F$-hyperbolic. 

\begin{theorem}\label{fincovcw}
Let $(\widehat M,\widehat d)$ and $(M,d)$ be compact metric spaces,
$\pi\colon\widehat M\to M$ be a finite covering map, $f\colon M\to M$ be a homeomorphism, and $\widehat f\colon \widehat M\to\widehat M$ be a homeomorphism that is a lift of $f$, i.e. $\pi\circ\widehat f=f\circ\pi$. The following holds:
\begin{enumerate}
\item If $f$ is continuum-wise expansive, then $\widehat f$ is continuum-wise expansive.
\item If $f$ is cw$_F$-expansive, then $\widehat f$ is cw$_F$-expansive.
\item If $f$ satisfies the cw-local-product-structure, then $\widehat f$ satisfies the cw-local-product-structure.
\end{enumerate}
In particular, if $f$ is cw$_F$-hyperbolic, then $\widehat f$ is cw$_F$-hyperbolic.
\end{theorem}

In the proof, we use the following classical result about finite covering maps.

\begin{theorem}[Eilenberg's Theorem - Thm 2.1.1 in \cite{AH}]
If $\pi\colon \widehat{M}\to M$ is a finite covering map of a compact metric space $M$, then there exist numbers $\lambda>0$ and $\theta>0$ such that for every subset $V\subset M$ with $\diam(V)<\lambda$, we have $$\pi^{-1}(V)=V_1\cup V_2\cup \cdots \cup V_k$$ with the following properties:

    \begin{itemize}
        \item $\pi\vert_{V_i}\colon V_i\to V$ is a homeomorphism for every $i\in\{1,\ldots,k\}$;
        \item If $i\neq i'$, then $\dist(V_i,V_{i'})>2\theta$;
        \item For each $\eps>0$, there exists $0<\delta<\lambda$ such that if $\diam_d(V)<\delta$, then $\diam_{\widehat d}(V_i)<\eps$ for every $i\in\{1,\ldots,k\}$.
    \end{itemize}
\end{theorem}

Now we prove the desired result.

\begin{proof}[Proof of Theorem \ref{fincovcw}]

(1) Let $c>0$ be a cw-expansivity constant of $f$. Since $\pi$ is uniformly continuous, there exists $\eta\in(0,c)$ such that $\widehat d(u,v)<\eta$ implies $d(\pi(u),\pi(v))<c$. We claim that $\eta$ is a cw-expansivity constant for $\widehat f$. Let $\widehat A\subset\widehat M$ be a continuum such that 
$$\operatorname{diam}_{\widehat d}(\widehat f^n(\widehat A))<\eta \,\,\,\,\,\, \text{for every} \,\,\,\,\,\, n\in\mathbb Z.$$
By the uniform continuity of $\pi$, we have
$$\operatorname{diam}_d(\pi(\widehat f^n(\widehat A)))<c\,\,\,\,\,\, \text{for every} \,\,\,\,\,\, n\in\mathbb Z$$
and since $\pi\circ\widehat f=f\circ\pi$, it follows that
$\pi(\widehat f^n(\widehat A))=f^n(\pi(\widehat A))$ and consequently
$$\operatorname{diam}_d(f^n(\pi(\widehat A)))<c\,\,\,\,\,\, \text{for every} \,\,\,\,\,\, n\in\mathbb Z.$$
Since the set $\pi(\widehat A)$ is a continuum in $M$ and $c$ is a cw-expansivity constant of $f$, we conclude that $\pi(\widehat A)=\{y\}$ for some $y\in M$, from which it follows that $\widehat A\subset\pi^{-1}(y)$. Since $\widehat A$ is a continuum and $\pi^{-1}(y)$ is finite, we conclude that
$\widehat A$ is a singleton and $\widehat f$ is continuum-wise expansive.

\vspace{+0.2cm}

(2) We assume that $f$ is cw$_F$-expansive and prove that $\widehat f$ is cw$_F$-expansive. Let $c>0$ be such that 
$$\#(C^s_c(x)\cap C^u_c(x))<\infty \quad \text{ for every } \quad x\in M$$
and choose $\eta\in(0,c)$ given by uniform continuity of $\pi$. 
Decrease $\eta$ when necessary, so that $\eta<\theta$, where $\theta$ is provided by Eilenberg's Theorem.
This way, we may assume that $\pi$ is injective on every subset of $\widehat M$ of diameter at most $2\eta$. In
particular, $\pi$ is injective on $C^s_\eta(y)\cap C^u_\eta(y)\subset\overline B(y,\eta)$, so
the finite cardinality is preserved by $\pi$.
Then
$$\#(C^s_{\eta}(y)\cap C^u_{\eta}(y))<\infty \quad \text{ for every } \quad y\in \widehat M$$
because whenever $z\in C^s_{\eta}(y)\cap C^u_{\eta}(y)$ we have that $\pi(z)\in C^s_{c}(\pi(y))\cap C^u_{c}(\pi(y))$. This proves that $\widehat f$ is cw$_F$-expansive.
\vspace{+0.2cm}

(3) Since $\pi:\widehat M\to M$ is a finite covering map over a compact metric space, by Eilenberg's Theorem there exist constants $\lambda>0$ and $\theta>0$ such that whenever $V\subset M$ satisfies $\operatorname{diam}(V)<\lambda$, the inverse image $\pi^{-1}(V)$ splits into pairwise disjoint sheets $U_k$ separated by a distance greater than $2\theta$.

We assume that $f$ satisfies the cw-local-product-structure and prove that $\widehat f$ satisfies the cw-local-product-structure. Given $\varepsilon>0$, choose $\rho\in(0,\min\{\varepsilon,\theta/3\})$. By Eilenberg's Theorem, there exists $\alpha\in(0,\min\{\rho,\lambda\})$ such that if $V\subset M$ satisfies $\operatorname{diam}(V)<\alpha$, then each sheet $U_k$ of $\pi^{-1}(V)$ has diameter smaller than $\rho$. Choose $\beta\in(0,\alpha/4)$ and $\delta'\in(0,\beta)$, given by the cw-local-product-structure of $f$, such that whenever $d(a,b)<\delta'$ we have
$C^s_\beta(a)\cap C^u_\beta(b)\neq\emptyset$. By uniform continuity of $\pi$, there exists $\eta\in(0,\delta')$ such that $\widehat d(x,y)<\eta$ implies $d(\pi(x),\pi(y))<\delta'$. Now choose $\widehat\delta\in(0,\min\{\eta,\theta/3\})$. We claim that if $\widehat d(x,y)<\widehat\delta$, then
$C^s_\varepsilon(x)\cap C^u_\varepsilon(y)\neq\emptyset$.

Indeed, let $x,y\in\widehat M$ be such that $\widehat d(x,y)<\widehat\delta$. Then $d(\pi(x),\pi(y))<\delta'$ and consequently there exists $z\in C^s_\beta(\pi(x))\cap C^u_\beta(\pi(y))$. Let $V=C^s_\beta(\pi(x))\cup C^u_\beta(\pi(y))$. The two continua meet at $z$, so $\operatorname{diam}(V)\leq4\beta<\alpha$. Since $\pi^{-1}(V)=V_1\cup\dots\cup V_k$ and $\pi\vert_{V_i}\colon V_i\to V$ is a homeomorphism for every $i\in\{1,\ldots,k\}$, we have $\pi^{-1}(C^s_\beta(\pi(x)))=S_1\cup\dots\cup S_k$ with $S_i\subset V_i$ satisfying 
$$\pi\vert_{V_i}(S_i)=C^s_\beta(\pi(x)) \,\,\,\,\,\, \text{for every} \,\,\,\,\,\, i\in\{1,\dots,k\}.$$
Similarly, $\pi^{-1}(C^u_\beta(\pi(y)))=U_1\cup\dots\cup U_k$ with $U_i\subset V_i$ satisfying 
$$\pi\vert_{V_i}(U_i)=C^u_\beta(\pi(y)) \,\,\,\,\,\, \text{for every} \,\,\,\,\,\, i\in\{1,\dots,k\}.$$
Let $S=S_i$ and $U=U_j$ be the ones containing $x$ and $y$, respectively. We prove that $S\subset C^s_\varepsilon(x)$ and $U\subset C^u_\varepsilon(y)$. Indeed,  
note that the choice of $\alpha$ ensures that $\diam_{\widehat d}(V_\ell)<\rho$ for every $\ell\in\{1,\dots,k\}$ and consequently $\diam_{\widehat d}(S)<\rho$ and $\diam_{\widehat d}(U)<\rho$. Also, $f^n(C^s_\beta(\pi(x)))$ has diameter at most $2\beta<\alpha$ for every $n\geq 0$. Since $\widehat f^n(S)$ is connected and
$$\pi(\widehat f^n(S))=f^n(\pi(S))\subset f^n(C^s_\beta(\pi(x))),$$ the set $\widehat f^n(S)$ is contained in a single sheet over $f^n(C^s_\beta(\pi(x)))$. Therefore 
$$\operatorname{diam}(\widehat f^n(S))<\rho<\varepsilon \,\,\,\,\,\, \text{for every} \,\,\,\,\,\, n\geq 0.$$ Since $S$ is connected and contains $x$, it follows that $S\subset C^s_\varepsilon(x)$. Similarly, using negative iterates, we obtain $U\subset C^u_\varepsilon(y)$.

It remains to prove that $S\cap U\neq\emptyset$. Since $z\in C^s_\beta(\pi(x))\cap C^u_\beta(\pi(y))$, there exists a unique point $\widehat z_s\in S$ such that $\pi(\widehat z_s)=z$, and a unique point $\widehat z_u\in U$ such that $\pi(\widehat z_u)=z$. In particular, $\widehat z_s$ and $\widehat z_u$ belong to the same fiber of $\pi$. Moreover, since $S$ and $U$ have diameter smaller than $\rho$, we have $\widehat d(\widehat z_s,x)<\rho$ and $\widehat d(\widehat z_u,y)<\rho$. Hence,
$$\widehat d(\widehat z_s,\widehat z_u)\leq \widehat d(\widehat z_s,x)+\widehat d(x,y)+\widehat d(y,\widehat z_u)<\rho+\widehat\delta+\rho<\theta.$$
Eilenberg's Theorem ensures that two distinct points in the same fiber of $\pi$ are separated by a distance greater than $2\theta$. Since $\widehat z_s$ and $\widehat z_u$ belong to the same fiber and are at a distance smaller than $\theta$, we must have $\widehat z_s=\widehat z_u$. Therefore, $S\cap U\neq\emptyset$.

Since $S\subset C^s_\varepsilon(x)$ and $U\subset C^u_\varepsilon(y)$, we conclude that
$C^s_\varepsilon(x)\cap C^u_\varepsilon(y)\neq\emptyset$. Since this argument can be done for each $\eps>0$, we proved that $\widehat f$ satisfies the cw-local-product-structure. This concludes the proof.
\end{proof}

\vspace{+0.2cm}

\hspace{-0.4cm}\textbf{Acknowledgments.}
The authors thank Alfonso Artigue and Lucas H. R. de Souza for discussions during the preparation of this work. They also thank the anonymous referees whose comments improved the precision and presentation of results. Bernardo Carvalho was supported by CNPq project number 446192/2024. 
Piotr Oprocha was supported by the project No.~CZ.02.01.01/00/23\_021/0008759 supported by EU funds through the Operational Programme Johannes Amos Comenius.
Piotr Oprocha is grateful to the Max Planck Institute for Mathematics in Bonn, and the University of Science and Technology of China in Hefei, for their hospitality and financial support. 

\vspace{+0.2cm}

\hspace{-0.4cm}\textbf{AI Statement.}
The results in the paper were developed and written down by the authors. During the revision of this manuscript, the authors used ChatGPT to improve the clarity of proofs, the grammar and sentence structure in selected sections of the text. The authors reviewed and revised all suggestions, verified technical accuracy, and take full responsibility for the final content.

%\vspace{1.0cm}

%{\em B. Carvalho}
%\vspace{0.2cm}

%\noindent

%National Laboratory for Scientific Computing – LNCC/MCTI 

%Av. Getúlio Vargas 333, CEP 25651-070, 

%Petrópolis – RJ, Brazil

%\vspace{0.2cm}

%\email{bmcarvalho@lncc.br}
%\vspace{1.0cm}

%{\em P. Oprocha}
%\vspace{0.2cm}

%\noindent

%Centre of Excellence IT4Innovations 

%Institute for Research and Applications of Fuzzy Modeling 

%University of Ostrava

%30. dubna 22, 701 03 Ostrava 1, Czech Republic

%\email{piotr.oprocha@osu.cz}
%\vspace{1.0cm}

%{\em R. Arruda and A. Sarmiento}
%\vspace{0.2cm}

%\noindent

%Departamento de Matem\'atica,

%Universidade Federal de Minas Gerais - UFMG

%Av. Ant\^onio Carlos, 6627 - Campus Pampulha

%Belo Horizonte - MG, Brazil.
%\vspace{0.2cm}

\end{document}